\documentclass[fleqn]{article}
\usepackage{amsmath}
\usepackage{amsthm}
\usepackage[pdftex]{graphicx}
\usepackage{exptex}
\usepackage{expthmi}
\usepackage{indentfirst}
\usepackage{color}
\usepackage{enumerate}
\usepackage[english]{babel}

\large\normalsize
\numberwithin{equation}{section}
\usepackage[toc,page]{appendix}

\theoremstyle{plain}
\newtheorem{thm}{Theorem}[section]
\newtheorem{lem}[thm]{Lemma}
\newtheorem{cor}[thm]{Corollary}
\newtheorem{prop}[thm]{Proposition}
\theoremstyle{definition}
\newtheorem{rem}[thm]{Remark}

\begin{document}
\title{ON STRATEGIC DEFENSE IN STOCHASTIC NETWORKS}
\author{Jewgeni H. Dshalalow\\
{\color{blue}eugene@fit.edu}
\and
Ryan White\\
{\color{blue}rwhite2009@fit.edu}}
\date{}
\maketitle

\vspace{-1cm}
\begin{center}
Department of Mathematical Sciences\\
College of Science\\
Florida Institute of Technology\\
Melbourne, Florida 32901, USA
\end{center}

\begin{abstract}
This paper deals with the detection and prediction of losses due to cyber attacks waged on vital networks. The accumulation of losses to a network during a series of attacks is modeled by a 2-dimensional monotone random walk process as observed by an independent delayed renewal process. The first component of the process is associated with the number of nodes (such as routers or operational sites) incapacitated by successive attacks. Each node has a weight associated with its incapacitation (such as loss of operational capacity or financial cost associated with repair), and the second component models the cumulative weight associated with the nodes lost. Each component has a fixed threshold, and crossing of a threshold by either component represents the network entering a critical condition. Results are given as joint functionals of the predicted time of the first observed threshold crossing along with the values of each component upon this time.


\end{abstract}

\section{\protect\centering Introduction}
\textbf{Background Information}. Complex large-scale networks form an integral part of our defense system and infrastructure: e.g. wireless communication networks, the Internet, multifunctional sensor networks, large-scale computer networks, and electrical grids. They play crucial roles in military and civilian installations and in biological research. A network can be regarded as a graph with nodes representing individual sensors, computers or servers, and so on, and edges connecting individual nodes. Such nodes are subject to failures of various origins, such as natural disasters, hostile attacks (e.g. cyber attacks), and benign hardware failures. Hostile attacks are prevalent instances of network disruptions and they are a growing threat to national security.

	In the past decade, we have observed intense cyber attacks, some of which have hindered proper operation of the computer networks of banks, private corporations, government facilities, and even vital infrastructure. In each situation, conventional security measures were unable to prevent such significant failures of security and the associated financial costs. Thus, it stands to reason to offer unconventional tools that may minimize the damage to networks. One of the common measures is to disconnect the intact portion of a network from the rest of it before incurring critical damage, but how does one recognize the threat and manage the network once an attack begins? In this article we lay out a foundation for predicting the time and caliber of potentially destructive damage to a network by means of statistics and operational calculus.
	
	Once an attack is launched, we seek to determine how much damage it will potentially cause and the time until a specified critical amount of damage is incurred. We model processes of attacks and cumulative losses of nodes and weights to predict both the time and size of critical damage. Further, we embellish the basic model by adding an auxiliary robustness threshold and analyze the fluctuation of the process about this intermediate threshold prior to entering a critical stage, which yields a more refined analysis of the nature of the accumulation of damage. This allows one to rank various risks to the network by considering the time between threshold crossings, which may, for example, be applied to early detection of hostile attacks and differentiation between attacks and benign hardware failures. Values of the thresholds can be determined heuristically, through optimization, or monitored upon the observation of incoming damage.
	
	\textbf{Modeling Networks.} As mentioned, networks can be modeled by graphs, and particularly weighted graphs, in which each edge or node is associated with a weight representing, for example, cost, bandwidth, distance, strength of social ties, energy, gravitational force, liquidity in financial networks, or probabilities. Random graphs are also commonly used, e.g. classical Erd\"os-R\a`enyi random graphs consist of $n$ vertices with $M$ edges chosen uniformly at random from all possible adjacencies, but we instead consider graphs with random node weights whose original structure can further be randomly altered.
	
	It is often thought that cyber attacks spread from one node to another along the graph as a branching process, but this is an incomplete picture for several reasons. Firstly, attacks often have multiple simultaneous targets (e.g. machines linked to a targeted router or servers housing virtualized machines) or victims quarantine groups of machines in response to attack detection, both of which result in batch losses. Further, the viral aspect of cyber attacks is significantly mitigated by firewalls, and viral attacks that do elude firewalls tend to practically immediately infect subnetworks (due to the structure of highly interconnected clusters that characterizes large-scale networks). Lastly, some of the most prevalent attacks are not necessarily viral at all (e.g. distributed denial-of-service attacks).
	
	Incoming damage to a network is subject to a complex process. An attack disables successive nodes and edges, effectively removing a subgraph with a random number of nodes with random associated weights (due to both the value of incident edges and intrinsic node value). Furthermore, such a network under attack is observed at potentially random epochs of time, as the losses become apparent. The spread of the damage is one of the most distinctive elements of a hostile action, so we predict its caliber using methods of stochastic analysis.
	
	\textbf{Related Literature.} Combating cyber crime involves various mathematical modeling and methods, as well as those of non-mathematical nature (such as prevention, physical and electronic defense mechanisms, and intelligence, which do not pertain to our article). Thus, we mention a few areas of current research on this topic and some literature known to the authors. Random and deterministic graph theory is probably the most common area for studying networking \lbrack 2, 3, 5, 10-15, 19, 23-27, 30-32\rbrack\ although just a few \lbrack 5, 23, 24\rbrack\ focused on potential cyber crime. Random walk and fluctuation analysis is an important hybrid in tools and modeling of interest, although the authors \lbrack 1, 6-9, 17-18, 28\rbrack\ did not have cyber crime in mind. There are also theoretical aspects of fluctuation analysis and applications to finance, physics, and other sciences that relate to the topic \lbrack 2, 6, 16, 20, 28, 29\rbrack.
	
	\textbf{The Layout of the Paper.} In Section 2, we begin with a model of a process where attacks occur according to a marked point process in time, with two-dimensional marks representing the cumulative nodes and weight lost up to time $t$ under delayed observation upon random time epochs. We derive a joint functional of each component as the process (node loss, weight loss, and time) at the first observed exit from $[0,M]\times[0,V]$ and at one observation prior.
	
	In Section 3, we extend the model to include an auxiliary robustness threshold for the node loss component, $M_1<M$. We derive a joint functional of the components of the process at both the first observed exit of the node component from $[0,M_1]$, the observed exit of the process from $[0,M]\times[0,V]$, and one observation prior to each for the ``confined'' process, under the restriction that the first observed $M_1$ crossing precedes the major crossing.
	
	The key technique in preserving the wide generality of our results in Sections 2 and 3 is to derive them under a composition of a number of operators of two types: the Laplace-Carson transform and so-called D-operator, the latter of which is introduced in some past work of one of the co-authors. Yet the flexibility and tractability of the inverses of these operators allows the reduction of the general formulas to fully explicit formulas for practical cases. We demonstrate this in Sections 4 and 5, where we will derive closed-form, explicit formulas for two sets of realistic probabilistic assumptions. Among other things, we derive explicit probabilistic results, including joint transforms of the process upon the $M_1$ and major crossings, which yield marginal transforms of each component at each observed crossing time and the time between the two observed crossings.

\section{\protect\centering A Basic Model}

Consider an infinite weighted graph in which weights are associated with nodes rather than edges. In reality there are not infinite graphs, but this assumption appropriately models large-scale networks. During a series of attacks, successive batches of nodes are incapacitated upon random time increments. Associated with each node is a random weight representing its value to the health of the network. We suppose the network enters a critical state wherein it may become dysfunctional if the number of nodes incapacitated by hostile attacks exceeds a fixed integer threshold $M$ or the magnitude of weights associated with the compromised nodes exceeds a fixed real threshold $W$. We proceed with more formalism of the model.

Let $\left(\Omega,\mathcal{F}\left(\Omega\right),P\right)$ be a probability space and let
\begin{equation}\label{a1}
\eta=\mathcal{N}\otimes\mathcal{W}=\sum\limits_{k\geq 1}\left(n_k,w_k\right)\varepsilon_{t_k}
\end{equation}
where $\varepsilon_a$ is a Dirac point measure, be a marked Poisson random measure on this probability space describing the evolution of damage taken to a network, where
\begin{align*}
&n_k\text{ nodes are destroyed at time }t_k,\,k=1,2,\ldots,
\\&w_k=\sum\limits_{j=1}^{n_k}w_{jk}\text{ is the nonnegative real weight associated with the $n_k$ nodes}
\end{align*}
and the underlying support counting measure $\sum^\infty_{k=1}\varepsilon_{t_k}$ is Poisson of rate $\lambda$ directed by $\lambda\left|\boldsymbol{\cdot}\right|$, where $\left|\boldsymbol{\cdot}\right|$ is the Borel-Lebesgue measure on $\mathcal{B}\left(\mathbb{R}_{+}\right)$.

We assume that $n_k$'s are iid$\,$(and independent of $w_{jk}$'s) with common marginal PGF (probability-generating function) $g\left(z\right)$, and $w_{jk}$'s are iid with common LST (Laplace-Stieltjes transform) $l\left(u\right)$ for $j,k\in\mathbb{N}$.

By straightforward probability arguments we obtain the following representation for $\eta$ in terms of its dependent components $\mathcal{N}$ and $\mathcal{W}$:
\begin{equation}\label{markedpoissontransform}
E\left[z^{\mathcal{N}\left(T\right)}e^{-v\mathcal{W}\left(T\right)}\right]=e^{\lambda\left|T\right|\left[g\left(zl\left(v\right)\right)-1\right]}\text{, Re}\left(v\right)\geq 0\text{, }|z|\leq 1
\end{equation}
where $T$ is a Borel subset of $\mathbb{R}_{+}$ and $\left|T\right|$ denotes the Borel-Lebesgue measure of set $T$.

Now, suppose $\eta$ is observed by a delayed renewal process
\begin{equation}
\mathcal{T}=\sum\limits_{n=0}^\infty\varepsilon_{\tau_n}
\end{equation}
and let
\begin{equation}
\Delta_n=\tau_n-\tau_{n-1},n\in\mathbb{N}\text{, be iid and independent of }\Delta_0=\tau_0
\end{equation}
such that
\begin{equation}
L_0\left(\theta\right)=E\left[e^{-\theta\Delta_0}\right]\text{, Re}\left(\theta\right)\geq 0\text{ (the LST of }\Delta_0=\tau_0)
\end{equation}
\begin{equation}
L\left(\theta\right)=E\left[e^{-\theta\Delta_1}\right]\text{, Re}\left(\theta\right)\geq 0\text{ (the common LST of }\Delta_n=\tau_n-\tau_{n-1}, n\in\mathbb{Z}_{> 0})
\end{equation}
Then, by Lemma A.1 (Appendix) and
\eqref{markedpoissontransform},
\begin{equation}\label{initialincrementbasic}
E\left[z^{\mathcal{N}\left([0,\tau_0]\right)}e^{-v\mathcal{W}\left([0,\tau_0]%
\right)-\theta\tau_0}\right]=L_0\left[\theta+\lambda-\lambda
g\left[zl\left(v\right)\right]\right]=\gamma_0\left(z,v,\theta\right)
\end{equation}
\begin{equation}\label{incrementbasic}
E\left[z^{\mathcal{N}\left((\tau_0,\tau_1]\right)}e^{-v\mathcal{W}\left((\tau_0,\tau_1]\right)-\theta\Delta_1}\right]=L\left[\theta+\lambda-\lambda g\left[zl\left(v\right)\right]\right]=\gamma\left(z,v,\theta\right)
\end{equation}
are the functionals describing the total number of lost nodes and their associated weights observed within time intervals $\left[0,\tau_0\right]$ and $(\tau_0,\tau_1]$, respectively. Since the increments  are \textit{iid}, the second corresponds to any $(\tau_{n-1},\tau_n]$.

Now, introduce a generic marked delayed renewal process
\begin{equation}\label{delayedrenewalprocess}
\mathcal{X}\otimes\mathcal{Y}\otimes\mathcal{T}=\sum\limits_{n=0}^\infty\left(X_n,Y_n\right)\varepsilon_{\tau_n}
\end{equation}
with mutually dependent components 

\begin{tabbing}
\hspace{1cm}\=\kill
\LTab{$\left(X_n,Y_n\right):\Omega\rightarrow\mathbb{N}\times\mathbb{R}_{+}$}
\end{tabbing}
whose relationship to the network will be explained later. Denote
\begin{equation}
N_n=\sum\limits_{i=0}^nX_i\text{ and }W_n=\sum\limits_{i=0}^nY_i.
\end{equation}
Introduce the random indices
\begin{equation}\label{discreteindexbasic}
\mu:=\inf\left\{n\geq 0:N_n>M\right\}\text{ for a fixed positive integer }M
\end{equation}
\begin{equation}\label{continuousindexbasic}
\nu:=\inf\left\{n\geq 0:W_n>V\right\}\text{ for a fixed positive real number }V
\end{equation}
called the exit indices.

We would say that the component $\mathcal{X}$ is terminated at time $\tau_\mu,$and component $\mathcal{Y}$ is terminated at time $\tau_\nu$ if $\mathcal{X}$ and $\mathcal{Y}$ acted alone, but we seek the first time either component terminates. If the original marked Poisson process $\eta$ is observed by a delayed renewal process $\mathcal{T}$, then the embedded process will exhibit (mutually dependent) increments $X_n$ and $Y_n$ as the marks in the process $\mathcal{X}\otimes\mathcal{Y}\otimes\mathcal{T}$. Such an observed process is regarded as ``terminated'' at time min$\left\{\tau_\mu,\tau_\nu\right\}$, the first observed passage time, which represents delayed information regarding the actual real-time crossing which occurred earlier.

First, we consider the confined process on trace $\sigma$-algebra $\mathcal{F}\left(\Omega\right)\cap\left\{\mu<\nu\right\}$, i.e. the process with component $\mathcal{X}$ being terminated first, and thus the first observed passage time $\tau_\mu$ will be the exit time by the confined process. Equation (2.\ref{delayedrenewalprocess}) will be modified as
\begin{equation}\left(\mathcal{X}\otimes\mathcal{Y}\otimes\mathcal{T}\right)_\mu=\sum\limits_{n=0}^\mu\left(X_n,Y_n\right)\varepsilon_{\tau_n}
\end{equation}
which gives a more precise definition of the process observed until $\tau_\mu$. We do the same for the confined processes on $\mathcal{F}(\Omega)\cap \{\mu=\nu\}$ and $\mathcal{F}(\Omega)\cap \{\mu>\nu\}$.

Then, we define the first observed passage index,
\begin{equation}\rho=\min\left\{\mu,\nu\right\}=\inf\left\{n:\left(N_n,W_n\right)\notin[0,M]\times[0,V]\right\}.
\end{equation}

Throughout the rest of this section, we consider various marginal and semi-marginal variants of the joint functional
\begin{align}\label{functionalbasic}
\varPhi&=\varPhi\left(\alpha_0,\alpha,\beta_0,\beta,h_0,h\right)=E\left[\alpha_0^{N_{\rho-1}}\alpha^{N_\rho}e^{-\beta_0W_{\rho-1}-\beta W_\rho}e^{-h_0\tau_{\rho-1}-h\tau_\rho}\right]\notag
\\&=
\begin{aligned}[t]
&E\left[\alpha_0^{N_{\mu-1}}\alpha^{N_\mu}e^{-\beta_0W_{\mu-1}-\beta W_\mu-h_0\tau_{\mu-1}-h\tau_\mu}\boldsymbol{1}_{\left\{\mu<\nu\right\}}\right]
\\&+E\left[\alpha_0^{N_{\mu-1}}\alpha^{N_\mu}e^{-\beta_0W_{\mu-1}-\beta W_\mu-h_0\tau_{\mu-1}-h\tau_\mu}\boldsymbol{1}_{\left\{\mu=\nu\right\}}\right]
\\&+E\left[\alpha_0^{N_{\nu-1}}\alpha^{N_\nu}e^{-\beta_0W_{\nu-1}-\beta W_\nu-h_0\tau_{\nu-1}-h\tau_\nu}\boldsymbol{1}_{\left\{\mu>\nu\right\}}
\right]
\end{aligned}\notag
\\&=\varPhi_{\mu<\nu}+\varPhi_{\mu=\nu}+\varPhi_{\mu>\nu}
\end{align}
of the observed process upon the first observed passage time $\tau_{\min\left\{\mu,\nu\right\}}$ and pre-observed passage time $\tau_{\min\left\{\mu,\nu\right\}-1}.\,$The latter is of particular interest due to the crudeness of the observed process.

The following transforms will be vital in the derivation of explicit results in the upcoming sections. Denote
\begin{equation}\label{doperatorbasic}
\text{D}_{pq}=\mathcal{L}\mathcal{C}_q\circ\mathcal{D}_p
\end{equation}
Here $\mathcal{L}\mathcal{C}_q$ is the Laplace-Carson transform:
\begin{equation}
\mathcal{L}\mathcal{C}_q\left(\boldsymbol{\cdot}\right)\left(y\right)=y\int_{q=0}^\infty e^{-yq}\left(\boldsymbol{\cdot}\right)dq\text{, Re}\left(y\right)>0
\end{equation}
with the inverse
\begin{equation}
\mathcal{L}\mathcal{C}_y^{-1}(\boldsymbol{\cdot})(q)=\mathcal{L}_y^{-1}\left(\boldsymbol{\cdot}\frac{1}{y}\right)\left(q\right)
\end{equation}
where $\mathcal{L}_y^{-1}$ is the inverse of the Laplace transform.

The operator $\mathcal{D}_p$ is defined as
\begin{equation}
\mathcal{D}_p\left(f\right)\left(x\right)=\left(1-x\right)\sum\limits_{p=0}^\infty x^pf\left(p\right),\left\|x\right\|<1
\end{equation}
where $\left\{f\left(p\right)\right\}$ is a sequence, with the inverse 
\begin{equation}
\mathcal{D}_x^k\left(\boldsymbol{\cdot}\right)=\lim_{x\rightarrow 0}\frac{1}{k!}\frac{\partial^r}{\partial x^r}\left[\frac{1}{1-x}\boldsymbol{\cdot}\right],\,k\geq 0\text{, and }\mathcal{D}_x^k=0\text{, for }k<0.
\end{equation}
such that $\mathcal{D}_x^k(\mathcal{D}_p\{f(p)\}) = f(k)$. Then, the full inverse operator is
\begin{equation}\label{doperatorinversebasic}
\text{D}^{-1}_{xy}\left(\boldsymbol{\cdot}\right)\left(p,q\right)=\mathcal{L}\mathcal{C}^{-1}_y\left(\mathcal{D}_x^p\left(\boldsymbol{\cdot}\right)\right)\left(q\right).
\end{equation}

The upcoming results for the joint transforms from \eqref{functionalbasic} will be derived under the inverse of composed operator D$_{pq}$. The utility of D$^{-1}_{xy}$ is at the heart of the derivation of joint and marginal transforms, which can yield results such as moments and distributions, of the components of process $\left(\mathcal{X}\otimes\mathcal{Y}\otimes\mathcal{T}\right)_\mu$ upon threshold crossings in a fully explicit form.

The below Theorem 2.1 establishes an analytically tractable$\,$formula for $\varPhi_{\mu<\nu}.\,$With \eqref{initialincrementbasic}-\eqref{incrementbasic}, we abbreviate
\begin{align}\label{firstbasicnotation}
&\gamma=\gamma\left(\alpha_0\alpha x,\beta_0+\beta+y,h_0+h\right)\\
&\gamma_0=\gamma_0\left(\alpha_0\alpha x,\beta_0+\beta+y,h_0+h\right)\\
&\varGamma=\gamma\left(\alpha x,\beta+y,h\right)\\
&\varGamma_0=\gamma_0\left(\alpha x,\beta+y,h\right)\\
&\varGamma^1=\gamma\left(\alpha,\beta+y,h\right)\\
&\label{lastbasicnotation}\varGamma^1_0=\gamma_0\left(\alpha,\beta+y,h\right)
\end{align}

\begin{thm}
In light of abbreviations \eqref{firstbasicnotation}-\eqref{lastbasicnotation}, the functional $\varPhi_{\mu<\nu}$ of the process on the trace $\sigma$-algebra $\mathcal{F}\left(\Omega\right)\cap\left\{\mu<\nu\right\}$ satisfies the following formula:
\begin{equation}
\varPhi_{\mu<\nu}=$D$_{xy}^{-1}\left(\varGamma_0^1-\varGamma_0+\frac{\gamma_0}{1-\gamma}\left(\varGamma^1-\varGamma\right)\right)\left(M,V\right)
\end{equation}
where D$^{-1}_{xy}$ is the inverse of operator D introduced in \eqref{doperatorbasic}.
\end{thm}
\begin{proof}
Introduce$\,$the families of indices
\begin{equation}
\mu\left(p\right)=\inf\left\{j:N_j>p\right\}
\end{equation}
\begin{equation}
\nu\left(q\right)=\inf\left\{k:W_k>q\right\}
\end{equation}
and
\begin{equation}
\varPhi_{\mu\left(p\right)<\nu\left(q\right)}=E\biggl[\alpha_0^{N_{\mu\left(p\right)-1}}\alpha^{N_{\mu\left(p\right)}}e^{-\beta_0W_{\mu\left(p\right)-1}-\beta W_{\mu\left(p\right)}-h_0\tau_{\mu\left(p\right)-1}-h\tau_{\mu\left(p\right)}}\boldsymbol{1}_{\left\{\mu\left(p\right)<\nu\left(q\right)\right\}}\biggr]
\end{equation}
In particular, we have
\begin{equation}
\varPhi_{\mu<\nu}=\varPhi_{\mu\left(M\right)<\nu\left(V\right)}.
\end{equation}

Application of D$_{pq}$ to $\varPhi_{\mu\left(p\right)\nu\left(q\right)}$ will bypass all terms except 1$_{\left\{\mu\left(p\right)<\nu\left(q\right)\right\}}$. Thus, after applying the operator D$_{pq}$ to random family $\left\{\boldsymbol{1}_{\{\mu\left(p\right)=j,\nu\left(q\right)=k\}}:p\geq 0,q\geq 0\right\}$ we arrive at
\begin{equation}\label{insensitivitybasic}
\text{D}_{pq}\left(\boldsymbol{1}_{\left\{\mu\left(p\right)=j,\nu\left(q\right)=k\right\}}\right)(x,y)=\left(x^{N_{j-1}}-x^{N_j}\right)\left(e^{-yW_{k-1}}-e^{-yW_k}\right)
\end{equation}
We first notice that
\begin{equation}
\textbf{1}_{\left\{\mu\left(p\right)=j,\nu\left(q\right)=k\right\}}=\left(\boldsymbol{1}_{\left\{N_{j-1}\leq p\right\}}\boldsymbol{1}_{\left\{N_j>p\right\}}\right)\left(\boldsymbol{1}_{\left\{W_{k-1}\leq q\right\}}\boldsymbol{1}_{\left\{W_k>q\right\}}\right)
\end{equation}
Then, we have
\begin{equation}
\text{D}_{pq}\left(\boldsymbol{1}_{\left\{\mu\left(p\right)=j,\nu\left(q\right)=k\right\}}\right)\left(x,y\right)=y\left(1-x\right)\sum\limits_{p=N_{j-1}}^{N_j-1}x^p\int_{q=W_{k-1}}^{W_k}e^{-yq}dq
\end{equation}
which yields \eqref{insensitivitybasic}. Denote
\begin{equation}
\varPsi\left(x,y\right)=\text{D}_{pq}\left(\varPhi_{\mu\left(p\right)\nu\left(q\right)}\right)\left(x,y\right).
\end{equation}
Since
\begin{equation}\label{insensitivitybasicfubinistep}
\varPhi_{\mu\left(p\right)<\nu\left(q\right)}=\sum\limits_{j\geq 0}\,\sum\limits_{k>j}E\biggl[\alpha_0^{N_{j-1}}\alpha^{N_j}e^{-\beta_0W_{j-1}-\beta W_j-h_0\tau_{j-1}-h\tau_j}\boldsymbol{1}_{\left\{\mu\left(p\right)=j,\,\nu\left(q\right)=k\right\}}\biggr]
\end{equation}
by Fubini's Theorem and \eqref{insensitivitybasicfubinistep}, we have
\begin{align}
\varPsi(x,y)&=\sum\limits_{j\geq 0}\sum\limits_{k>j}E\biggl[\alpha_0^{N_{j-1}}\alpha^{N_j}e^{-\beta_0W_{j-1}-\beta W_j-h_0\tau_{j-1}-h\tau_j}\left(x^{N_{j-1}}-x^{N_j}\right)\left(e^{-yW_{k-1}}-e^{-WB_k}\right)\biggr]\notag
\\&=\sum\limits_{j\geq
0}\,\sum\limits_{k>j}R_{1j}R_{2j}R_{3jk}R_{4k}
\end{align}
where
\begin{equation}
R_{1j}=E\biggl[\left(\alpha_0\alpha x\right)^{N_{j-1}}e^{-\left(\beta_0+\beta+y\right)W_{j-1}-\left(h_0+h\right)\tau_{j-1}}\biggr]=\left\{
\begin{matrix}
1,&j=0\\
\gamma_0\gamma^{j-1},&j>0
\end{matrix}\right.
\end{equation}
\begin{equation}
R_{2j}=E\biggl[\alpha^{X_j}\left(1-x^{X_j}\right)e^{-\left(\beta+y\right)Y_j-h\Delta_j}\biggr]=\left\{
\begin{matrix}
\varGamma_0^1-\varGamma_0,&j=0\\
\varGamma^1-\varGamma,&j>0.
\end{matrix}\right.
\end{equation}

We observe that the marginal distribution of $Y_i$'s for $i>\mu\,(=j)$ can be different from $Y_1,\ldots,Y_\mu,$ because after a number of nodes in excess of $M$ are purged at $\tau_\mu$, the forthcoming events (such as further damage to the network leading to a sure excess of weight above $V$ are so far limited to those on the sub-$\sigma$-algebra $\mathcal{F}(\Omega)\cap\left\{\mu<\nu\right\}$). The corresponding marginal transform of those $Y_i$'s can differ from $\gamma$ in the next two equations. However, this will not alter the result of the summation of $R_{3jk}R_{4k}$, as we will see. Without loss of generality we omit the corresponding formalism thus having
\begin{align}
&R_{3jk}=E\biggl[e^{-y\left(Y_{j+1}+\ldots+Y_{k-1}\right)}\biggr]=%
\gamma^{k-1-j}\left(1,y,0\right),\,k>j\geq
0
\\&R_{4k}=E\biggl[1-e^{-yY_k}\biggr]=1-\gamma\left(1,y,0\right),k>j\geq
0
\end{align}

At this moment, we will suspend the proof to derive and state some results necessary to the completion of the proof of Theorem 2.1.

\begin{thm}
The norm
\begin{equation}
\left\|L\left[\theta+\lambda-\lambda g\left[zl\left(v\right)\right]\right]\right\|=\left\|\gamma\left(z,v,\theta\right)\right\|<1
\end{equation}
if
\begin{equation}
\text{Re}\left(\theta\right)>0,\,1>\left\|z\right\|,\text{ and }\text{Re}\left(v\right)>0
\end{equation}
where any two of the inequalities can be weakened to $\geq$.
\end{thm}
\begin{proof}
Consider
\begin{equation}
L\left(\vartheta\right)=E\left[e^{-\vartheta\Delta_1}\right]=\int_{v%
\geq 0}e^{-\vartheta
v}P_{\Delta_1}\left(dv\right)
\end{equation}
From (2.44)
\begin{align}
\left\|L\left(\vartheta\right)\right\|&\leq\int_{v\geq 0}\left\|e^{-\vartheta v}\right\|P_{\Delta_1}\left(dv\right)=\int_{v\geq 0}e^{-\text{Re}\left(\vartheta\right)v}P_{\Delta_1}\left(dv\right)\notag
\\&\leq\int_{v=0}^1P_{\Delta_1}\left(dv\right)+e^{-\text{Re}\left(\vartheta\right)}\int_{v\geq 1}P_{\Delta_1}\left(dv\right)=c+\left(1-c\right)e^{-\text{Re}\left(\vartheta\right)}
\end{align}
where $c=P\left\{\Delta_1\leq 1\right\}$. If
\begin{equation}
c+\left(1-c\right)e^{-\text{Re}\left(\vartheta\right)}<1
\end{equation}
then $\left\|L\left(\vartheta\right)\right\|<1$, so
Re$\left(\vartheta\right)>0$ is sufficient for
$\left\|L\left(\vartheta\right)\right\|<1$. Also, for
$\vartheta=\theta+\lambda-\lambda g\left[zl\left(v\right)\right]$, we have 
\begin{equation}
e^{-\text{Re}\left(\vartheta\right)}=e^{-\text{Re}\left(\theta\right)}e^{-\lambda[1-\text{Re}\left(g\left(zl\left(v\right)\right)\right)]}<1
\end{equation}
Inequality (2.47) holds if each of the two factors in (2.48) is less than one, or in the weaker form, one of the factors equal to one and one is strictly less than one. Let us assume that each factor is less than one. Then, we have Re$\left(\theta\right)>0$ and Re$\left(g\left(zl\left(v\right)\right)\right)<1$.

Since Re$\left(g\left(zl\left(v\right)\right)\right)\leq\left\|g\left(zl\left(v\right)\right)\right\|$, by imposing $\left\|g\left(zl\left(v\right)\right)\right\|<1$ we ensure that Re$\left(g\left(zl\left(v\right)\right)\right)<1$, which is true by the Schwarz lemma below if $\left\|zl\left(v\right)\right\|<1$, as follows.
\end{proof}
\begin{thm}[Schwarz Lemma]
Let $g\left(z\right)$ be an analytic function inside the unit ball $B\left(0,1\right)$ and satisfies the condition $\left\|g\left(z\right)\right\|\leq 1$ and $g\left(0\right)=0$. Then $\left\|g\left(z\right)\right\|\leq\left\|z\right\|$ and $\left\|g'\left(0\right)\right\|\leq 1$. If $\left\|f\left(z\right)\right\|=\left\|z\right\|$ for some $z\neq 0,$ then $f\left(z\right)=cz,$where $c$ is a complex-valued constant of modulus $1$.
\end{thm}

Indeed, by the Schwarz lemma whose conditions are obviously met with $g\left(0\right)=0$, $\|g\left(zl\left(v\right)\right)\|\leq\|zl\left(v\right)\|$, and hence $\left\|g\left(zl\left(v\right)\right)\right\|<1$ if $\left\|zl\left(v\right)\right\|<1$ holds. We first show that $\left\|l\left(v\right)\right\|<1$ under the condition that Re$\left(v\right)>0$.

Proceeding with $\left\|l\left(v\right)\right\|$ exactly as with $L\left(\vartheta\right)$ we impose Re$\left(v\right)>0$ in order to have $\left\|l\left(v\right)\right\|<1$. Finally, $\left\|z\right\|\leq 1$ is the common domain for a pgf like $g\left(z\right)$. Imposing $\left\|z\right\|<1$ we can relax Re$\left(l\left(v\right)\right)$ to be $\geq 0$ or Re$\left(\theta\right)\geq 0$.\hfill{\qedsymbol}

Notice that upon application of operator D$_{pq}$ of (2.15), the output variables must be restricted to $\left\|\cdot\right\|<1$. Therefore, with Re$\left(\theta\right)\geq 0$, we have convergence. Also observe that Theorem 2.2 can be applied to the norm of $\gamma_0$ of \eqref{initialincrementbasic} under minor modifications.

Now we continue with the proof of Theorem 2.1. Summing up $\,\sum_{k>j}R_{3jk}R_{4k}\,$yields 1 (with $||\gamma\left(1,y,0\right)||<1$, under a minor sufficient condition given in Theorems 2.2 and 2.3). $\sum_{j\geq 0}R_{1j}R_{2j}$ converges by the same argument, yielding
\begin{equation}
\varPsi\left(x,y\right)=\varGamma_0^1-\varGamma_0+\frac{\gamma_0}{1-%
\gamma}\left(\varGamma^1-\varGamma\right).
\end{equation}
Applying the inverse operator D$_{xy}^{-1}$ of \eqref{doperatorinversebasic} to $\varPsi\left(x,y\right)$ yields the statement of Theorem 2.1.
\end{proof}

Proceeding as in Theorem 2.1, we can find the functionals $\varPhi_{\mu=\nu}$ and $\varPhi_{\mu>\nu}$, using some additional notation:
\begin{align}
&\label{firstnotationbasic2}\zeta=\gamma\left(\alpha x,\beta,h\right)
\\&\zeta_0=\gamma_0\left(\alpha x,\beta,h\right)
\\&\zeta^1=\gamma\left(\alpha,\beta,h\right)
\\&\label{lastnotationbasic2}\zeta^1_0=\gamma_0\left(\alpha,\beta,h\right)
\end{align}

\begin{prop}
In light of abbreviations \eqref{firstbasicnotation}-\eqref{lastbasicnotation} and \eqref{firstnotationbasic2}-\eqref{lastnotationbasic2}, the functional $\varPhi_{\mu=\nu}$ of the process on the trace $\sigma$-algebra $\mathcal{F}\left(\Omega\right)\cap\left\{\mu=\nu\right\}$ satisfies the following formula:
\begin{equation}
\varPhi_{\mu=\nu}=$D$_{xy}^{-1}\left(\zeta_0^1-\zeta_0-\varGamma^1_0+\varGamma_0+\frac{\gamma_0}{1-\gamma}\left(\zeta^1-\zeta-\varGamma^1+\varGamma\right)\right)\left(M,V\right)
\end{equation}
\end{prop}

\begin{prop}
In light of abbreviations \eqref{firstbasicnotation}-\eqref{lastbasicnotation} and \eqref{firstnotationbasic2}-\eqref{lastnotationbasic2}, the functional $\varPhi_{\mu>\nu}$ of the process on the trace $\sigma$-algebra $\mathcal{F}\left(\Omega\right)\cap\left\{\mu>\nu\right\}$ satisfies the
following formula:
\begin{equation}
\varPhi_{\mu>\nu}=$D$_{xy}^{-1}\left(\zeta_0-\varGamma_0+\frac{\gamma_0}{1-\gamma}\left(\zeta-\varGamma\right)\right)\left(M,V\right)
\end{equation}
\end{prop}

By the linearity of the inverse operator and expectation, the previous 3 results yield the functional of the process no longer confined to a particular ordering of the exit indices:

\begin{cor}
In light of abbreviations \eqref{firstbasicnotation}-\eqref{lastbasicnotation} and \eqref{firstnotationbasic2}-\eqref{lastnotationbasic2}), the functional $\varPhi$ of the process on the $\sigma$-algebra $\mathcal{F}\left(\Omega\right)$ satisfies the following formula:
\begin{equation}
\varPhi=$D$_{xy}^{-1}\left(\zeta_0^1-\varGamma_0+\frac{\gamma_0}{1-\gamma}\left(\zeta^1-\varGamma\right)\right)\left(M,V\right)
\end{equation}
\end{cor}

Note that the basic model in this section, in a much more rudimentary form was proposed and studied in Dshalalow \lbrack 7\rbrack\ and further in Dshalalow and Liew \lbrack 9\rbrack.

\section{\protect\centering An Auxiliary Threshold for the Discrete Component}

\textbf{Preliminaries.} In this section we will analyze the process introduced in Section 2 more scrupulously. We will add an intermediate control level $M_1<M$ and incorporate it into the \ functional $\varPhi_{\mu<\nu}\,$of (2.26). The information associated with the $M_1$-threshold will become more conclusive on what leads to critical threshold crossings of the network. The idea of an auxiliary threshold stems from Dshalalow and Ke \lbrack 8\rbrack\ applied there to stochastic games, with different utility and development compared to the present paper.

We define the corresponding exit indices
\begin{align}
&\mu_1=\inf\left\{j:N_j>M_1\right\}
\\&\mu=\inf\left\{k:N_k>M\right\}
\\&\nu=\inf\left\{n:W_n>V\right\}
\end{align}

A realization of process $\mathcal{X}\otimes\mathcal{Y}\otimes\mathcal{T}$ (of losses defined in \eqref{delayedrenewalprocess}) in Figure 1 illustrates how it operates with respect to the introduced main and auxiliary thresholds. We can regard $\mathcal{X}\otimes\mathcal{Y}\otimes\mathcal{T}$ as a two-dimensional random walk on a random grid (rather than traditional lattice). We have a rectangular region formed of rectangular sectors in white, green, and red colors. In real-time a ``particle'' tries to escape the white-green area at the first opportune time when the cumulative loss of nodes exceeds $M$ or the cumulative weight loss exceeds $V$, whichever comes first. It leaves a polygonal path in blue and a cruder, observed, path ingreen. The particle enters the green area indicating that a lower threshold $M_1$ is crossed, while neither $M$ nor $V$ was. In reality, the green area can be empty with a positive probability.

In the figure, the underlying real-time process (the blue dots) represents the real-time incoming damages, which are observed only upon $\tau_k$'s (depicted by the green dots), where the $M_1$-observed-crossing occurs before the first observed passage time (FOPT) of $M$ or $V$ (i.e. there is an observation in the green area), at which the components of the process may or may not coincide with their values at the real-time FPT.

\begin{center}
\includegraphics[scale=0.5]{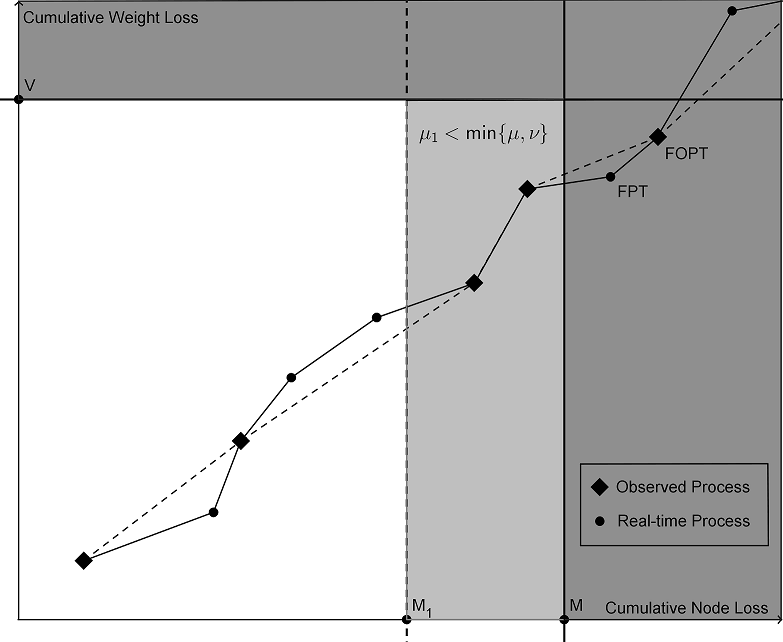}
\end{center}

Here, the operator we will use contains an additional discrete transform $\mathcal{D}_q$ corresponding to the $M_1$ threshold.
\begin{equation}\label{auxdoperator}
\text{D}_{pqs}\left(\boldsymbol{\cdot}\right)\left(x,y,w\right)=\left(1-x\right)\left(1-y\right)w\sum\limits_{p=0}^\infty x^p\sum\limits_{q=0}^\infty y^q\int_{s=0}^\infty e^{-ws}\left(\boldsymbol{\cdot}\right)ds
\end{equation}
where $\left|x\right|<1,\left|y\right|<1,$Re$\left(w\right)>0$, i.e. D$_{pqs}=\mathcal{L}\mathcal{C}_s\circ\mathcal{D}_q\circ\mathcal{D}_p$. Analogous to the operator in Section 2, D$_{xyw}^{-1}$ denotes the inverse of D$_{xyw}$.

The following result seems surprising. Unlike Section 2, with independent thresholds $p$ and $q$ for two active components, now we have two of the three thresholds dependent (since $M_1<M$) and yet the below lemma asserts that the application of operator D$_{pqs}$ to the analogous indicator produces the same result as the case with no relationship between variable levels $p$ and $q$ (representing $M_1$ and $M$, respectively). This is a critical asset for our strategy of obtaining closed-form functionals.

\begin{lem}[$M_1$-Level Insensitivity]
\begin{align}
&\text{D}_{pqs}\left(\boldsymbol{1}_{\left\{\mu_1\left(p,q\right)=j,\mu\left(p,q\right)=k,\nu\left(p,q,s\right)=n\right\}}\boldsymbol{1}_{\left\{j<k<n\right\}}\right)(x,y,w)\notag
\\&=\left(x^{N_{j-1}}-x^{N_j}\right)\left(y^{N_{k-1}}-y^{N_k}\right)\left(e^{-wW_{n-1}}-e^{-wW_n}\right)\boldsymbol{1}_{\left\{j<k<n\right\}}
\end{align}
\end{lem}
\begin{proof}
Let
\begin{align}
&\mu_1\left(p,q\right):=\inf\left\{j:N_j>p\right\}
\\&\mu\left(p,q\right):=\inf\left\{k:N_k>q\right\}
\\&\nu\left(p,q,s\right):=\inf\left\{n:W_n>s\right\}
\end{align}
We have
\begin{align*}
&\boldsymbol{1}_{\left\{\mu_1\left(p,q\right)=j,\mu\left(p,q\right)=k,\nu\left(p,q,s\right)=n\right\}}
\\&=\boldsymbol{1}_{\left\{p<q\right\}}\boldsymbol{1}_{\left\{N_{j-1}\leq p\right\}}\boldsymbol{1}_{\left\{N_j>p\right\}}\boldsymbol{1}_{\left\{N_{k-1}\leq q\right\}}\boldsymbol{1}_{\left\{N_k>q\right\}}\boldsymbol{1}_{\left\{W_{n-1}\leq s\right\}}\boldsymbol{1}_{\left\{W_n>s\right\}}\boldsymbol{1}_{\left\{j<k<n\right\}}
\\&=\boldsymbol{1}_{\left\{N_{j-1}\leq p\right\}}\boldsymbol{1}_{\left\{N_j>p\right\}}\left(\boldsymbol{1}_{\left\{N_k>q\right\}}-\boldsymbol{1}_{\left\{N_{k-1}>q\right\}}\right)\boldsymbol{1}_{\left\{W_{n-1}\leq s\right\}}\boldsymbol{1}_{\left\{W_n>s\right\}}\boldsymbol{1}_{\left\{p<q\right\}}\boldsymbol{1}_{\left\{j<k<n\right\}}
\end{align*}
Applying operator \eqref{auxdoperator}) we have
\begin{align*}
&\text{D}_{pqs}\left(\boldsymbol{1}_{\left\{\mu_1\left(p,q\right)=j,\mu\left(p,q\right)=k,\nu\left(p,q,s\right)=n\right\}}\right)\left(x,y,w\right)
\\&\hspace{1cm}=\left(1-x\right)\left(1-y\right)w\,\sum\limits_{p=N_{j-1}}^{N_j-1}x^p\,\sum\limits_{q=N_{k-1}}^{N_k-1}y^q\int_{s=W_{n-1}}^{W_n}e^{-ws}ds,
\end{align*}
which readily yields the statement of the lemma.
\end{proof}

The extended functional,
\begin{align}
\varPhi_{\mu_1<\mu<\nu}&=\varPhi_{\mu_1<\mu<\nu}\left(u_0,u,\alpha_0,\alpha,v_0,v,\beta_0,\beta,\theta_0,\theta,h_0,h\right)\notag
\\&=E\biggl[u_0^{N_{\mu_1-1}}u^{N_{\mu_1}}\alpha_0^{N_{\mu-1}}\alpha^{N_\mu}e^{-v_0W_{\mu_1-1}-vW_{\mu_1}-\beta_0W_{\mu-1}-\beta W_\mu}\notag
\\&\hspace{1cm}\times e^{-\theta_0\tau_{\mu_1-1}-\theta\tau_{\mu_1}-h_0\tau_{\mu-1}-h\tau_\mu}\boldsymbol{1}_{\left\{\mu_1<\mu<\nu\right\}}\biggr]
\end{align}
carries refined information about the cumulative damage upon crossing $M_1$, as well as upon the other reference times. Theorem 3.2 below will establish an explicit formula for $\Phi_{\mu_1<\mu<\nu}$.

Having \eqref{initialincrementbasic}-\eqref{incrementbasic} in mind and replacing the abbreviations of Section 2, we denote
\begin{align}
&\label{auxfirstnotation}\varphi=\gamma\left(u_0u\alpha_0\alpha xy,v_0+v+\beta_0+\beta+w,\theta_0+\theta+h_0+h\right)
\\&\varphi_0=\gamma_0\left(u_0u\alpha_0\alpha xy,v_0+v+\beta_0+\beta+w,\theta_0+\theta+h_0+h\right)
\\&\phi=\gamma\left(u\alpha_0\alpha xy,v+\beta_0+\beta+w,\theta+h_0+h\right)
\\&\phi_0=\gamma_0\left(u\alpha_0\alpha xy,v+\beta_0+\beta+w,\theta+h_0+h\right)
\\&\phi^1=\gamma\left(u\alpha_0\alpha y,v+\beta_0+\beta+w,\theta+h_0+h\right)
\\&\phi^1_0=\gamma_0\left(u\alpha_0\alpha y,v+\beta_0+\beta+w,\theta+h_0+h\right)
\\&\psi=\gamma\left(\alpha_0\alpha y,\beta_0+\beta+w,h_0+h\right)
\\&\chi=\gamma\left(\alpha y,\beta+w,h\right)
\\&\chi^1=\gamma\left(\alpha,\beta+w,h\right)
\\&\xi=\gamma\left(\alpha y,\beta,h\right)
\\&\label{auxlasnotation}\xi^1=\gamma\left(\alpha,\beta,h\right)
\end{align}

\begin{thm}
The functional $\varPhi_{\mu_1<\mu<\nu}$ of the of the network damage on the trace $\sigma$-algebra $\mathcal{F}\left(\Omega\right)\cap\left\{\mu_1<\mu<\nu\right\}$ satisfies the following formula under notation
\eqref{auxfirstnotation}-\eqref{auxlasnotation}:
\begin{equation}
\varPhi_{\mu_1<\mu<\nu}=$D$_{xyw}^{-1}\left(\left[\phi_0^1-\phi_0+\frac{\varphi_0}{1-\varphi}(\phi^1-\phi)\right]\frac{\chi^1-\chi}{1-\psi}\right)\left(M_1,M,V\right)
\end{equation}
\end{thm}
\begin{proof}
As in the proof of Theorem 2.1, we will apply an analog of the LCD operator \eqref{auxdoperator} to
\begin{equation}
\left\{\boldsymbol{1}_{\left\{\mu_1\left(p,q\right)=j,\mu\left(p,q\right)=k,\nu\left(p,q,s\right)=n\right\}}:\left(p,q\right)\in\mathbb{N},s\in\mathbb{R}_{+}\right\}
\end{equation}
\end{proof}
The functional $\varPhi_{\mu_1<\mu<\nu}$ becomes
\begin{align}
&\varPhi_{\mu_1\left(p,q\right)<\mu\left(p,q\right)<\nu\left(p,q,s\right)}\notag
\\&=E\biggl[u_0^{N_{\mu_1\left(p,q\right)-1}}u^{N_{\mu_1\left(p.q\right)}}\alpha_0^{N_{\mu\left(p,q\right)-1}}\alpha^{N_{\mu\left(p,q\right)}}e^{-v_0W_{\mu_1\left(p,q\right)-1}-vW_{\mu_1\left(p,q\right)}}\notag
\\&\hspace{1cm}\times e^{-\beta_0W_{\mu\left(p,q\right)-1}-\beta W_{\mu\left(p,q\right)}-\theta_0\tau_{\mu_1\left(p,q\right)-1}-\theta\tau_{\mu_1\left(p,q\right)}-h_0\tau_{\mu\left(p,q\right)-1}-h\tau_{\mu\left(p,q\right)}}\notag
\\&=E\biggl[u_0^{N_{\mu_1\left(p,q\right)-1}}u^{N_{\mu_1\left(p.q\right)}}\alpha_0^{N_{\mu\left(p,q\right)-1}}\alpha^{N_{\mu\left(p.q\right)}}e^{-v_0W_{\mu_1\left(p,q\right)-1}-vW_{\mu_1\left(p,q\right)}}\notag
\\&\hspace{1cm}\times\boldsymbol{1}_{\left\{\mu_1\left(p,q\right)<\mu\left(p,q\right)<\nu\left(p,q,s\right)\right\}}\biggr]
\end{align}
with
\begin{equation}
\varPhi_{\mu_1<\mu<\nu}=\varPhi_{\mu_1\left(M_1,M\right)<\mu\left(M_1,%
M\right)<\nu\left(M_1,M,V\right)}
\end{equation}
We first arrive at
\begin{align*}
&\text{D}_{pqs}\left(\boldsymbol{1}_{\left\{\mu_1\left(p,q\right)=j,\mu\left(p,q\right)=k,\nu\left(p,q,s\right)=n\right\}}\boldsymbol{1}_{\left\{j<k<n\right\}}\right)\left(x,y,w\right)
\\&\hspace{1cm}=\left(x^{N_{j-1}}-x^{N_j}\right)\left(y^{N_{k-1}}-y^{N_k}\right)\left(e^{-wW_{n-1}}-e^{-wW_n}\right)\boldsymbol{1}_{\left\{j<k<n\right\}}
\end{align*}
as per Lemma\label{q} 3.1. Denote
\begin{equation}
\varPsi\left(x,y,w\right):=$D$_{pqs}\left(\varPhi_{\mu_1\left(p,q%
\right)<\mu\left(p,q\right)<\nu\left(p,q,s\right)}\right)\left(x,y,w\right).
\end{equation}
Analogous to the proof in Theorem 2.1, we have
\begin{align}
\varPsi\left(x,y,w\right)&=\sum\limits_{j\geq 0}\,\,\sum\limits_{k>j}\,\,\sum\limits_{n>k}E\biggl[u_0^{N_{j-1}}u^{N_j}\alpha_0^{N_{k-1}}\alpha^{N_k}e^{-v_0W_{j-1}-vW_j-\beta_0W_{k-1}-\beta W_k}\notag
\\&\hspace{2cm} \times e^{-v_0W_{j-1}-vW_j-\beta_0W_{k-1}-\beta W_k-\theta_0\tau_{j-1}-\theta\tau_j-h_0\tau_{k-1}-h\tau_k}\notag
\\&\hspace{2cm}\times\left(x^{N_{j-1}}-x^{N_j}\right)\left(y^{N_{k-1}}-y^{N_k}\right)\left(e^{-wW_{n-1}}-e^{-wW_n}\right)\biggr]\notag
\\&=\sum\limits_{j\geq 0}\,\,\sum\limits_{k>j}\,\,\sum\limits_{n>k}R_{1j}R_{2j}R_{3jk}R_{4k}R_{5kn}
\end{align}
where
\begin{align}
R_{1j}&=E\left[\left(u_0u\alpha_0\alpha xy\right)^{N_{j-1}}e^{-\left(v_0+v+\beta_0+\beta+w\right)W_{j-1}-\left(\theta_0+\theta+h_0+h\right)\tau_{j-1}}\right]\notag
\\&=\left\lbrace\begin{matrix}
1,&j=0\\
\varphi_0\varphi^{j-1},&j>0
\end{matrix}
\right.
\end{align}
\begin{align}
R_{2j}&=E\left[\left(u\alpha_0\alpha y\right)^{X_j}\left(1-x^{X_j}\right)e^{-\left(v+\beta_0+\beta+w\right)Y_j-\left(\theta+h_0+h\right)\Delta_j}\right]\notag
\\&=\left\{
\begin{matrix}
\phi_0^1-\phi_0,&j=0\\
\phi^1-\phi,&j>0
\end{matrix}
\right.
\end{align}
\begin{align}
R_{3jk}&=E\biggl[\left(\alpha_0\alpha
y\right)^{X_{j+1}+\ldots+X_{k-1}}e^{-\left(\beta_0+\beta+w\right)\left(Y_{j+1}+\ldots+Y_{k-1}\right)}e^{-\left(h_0+h\right)\left(\Delta_{j+1}+\ldots+\Delta_{k-1}\right)}\biggr]\notag
\\&=\psi^{k-1-j},\,k>j\geq 0
\end{align}
\begin{equation}
R_{4k}=E\left[\alpha^{X_k}\left(1-y^{X_k}\right)e^{-\left(\beta+w\right)Y_k-h\Delta_k}\right]=\chi^1-\chi
\end{equation}
\begin{align}
R_{5kn}&=E\left[e^{-w\left(Y_{k+1}+\ldots+Y_{n-1}\right)}\right]E\left[1-e^{-wY_n}\right]\notag
\\&=\gamma^{n-1-k}\left(1,w,0\right)\left(1-\gamma\left(1,w,0\right)\right),n>k
\end{align}
Using Theorem 2.2, $\sum\limits_{n>k}R_{5kn}=1$. The other summations converge similarly, so
\begin{equation}\label{last1}
\varPsi\left(x,y,z\right)=\biggl[\phi_0^1-\phi_0+\frac{\varphi_0}{1-\varphi}\left(\phi^1-\phi\right)\biggr]\frac{\chi^1-\chi}{1-\psi}
\end{equation}
The assertion of the theorem follows upon application of operator
D$_{xyw}^{-1}\,$to \eqref{last1}.\hfill{\qedsymbol}

\textbf{The Complete Auxiliary Functional.} The formula proved in Theorem 3.2 is itself useful. However, we are further interested in information on the process restricted to the case where the first observed crossing of $M_1$ precedes the first observed crossing of $M$ or $V$, whichever comes first: i.e. a functional restricted to $\left\{\mu_1<\min\left\{\mu,\nu\right\}\right\}$:
\begin{align}
&\varPhi_{\mu_1<\min\left\{\mu,\nu\right\}}\left(u_0,u,\alpha_0,\alpha,v_0,v,\beta_0,\beta,\theta_0,\theta,h_0,h\right)\notag
\\&=E\biggl[u_0^{N_{\mu_1-1}}u^{N_{\mu_1}}\alpha_0^{N_{\text{min}\left\{\mu,\nu\right\}-1}}\alpha^{N_{\text{min}\left\{\mu,\nu\right\}}}e^{-v_0W_{\mu_1-1}-vW_{\mu_1}-\beta_0W_{\text{min}\left\{\mu,\nu\right\}-1}-\beta W_{\text{min}\left\{\mu,\nu\right\}}}\notag
\\&\hspace{1cm}\times e^{-\theta_0\tau_{\mu_1-1}-\theta\tau_{\mu_1}-h_0\tau_{\text{min}\left\{\mu,\nu\right\}-1}-h\tau_{\text{min}\left\{\mu,\nu\right\}}}\boldsymbol{1}_{\{\mu_1<\text{min}\left\{\mu,\nu\right\}\}}\biggr]\notag
\\&=\varPhi_{\mu_1<\mu<\nu}+\varPhi_{\mu_1<\mu=\nu}\,+\,\varPhi_{\mu_1<\nu<\mu}
\end{align}

We derived the first functional in Theorem 3.2, and we find the others in the following two results whose proofs are analogous to Theorem 3.2.

\begin{prop}
The functional $\varPhi_{\mu_1<\mu=\nu}$ of the of the network damage on the trace $\sigma$-algebra $\mathcal{F}\left(\Omega\right)\cap\left\{\mu_1<\mu=\nu\right\}\,$satisfies the following formula \ under notation \eqref{auxfirstnotation}-\eqref{auxlasnotation}:
\begin{equation}
\varPhi_{\mu_1<\mu=\nu}=$D$_{xyw}^{-1}\left(\left[\phi_0^1-\phi_0+\frac{\varphi_0}{1-\varphi}\left(\phi^1-\phi\right)\right]\frac{\xi^1-\xi+\chi-\chi^1}{1-\psi}\right)\left(M_1,M,V\right).
\end{equation}
\end{prop}

\begin{prop}
The functional $\,\varPhi_{\mu_1<\nu<\mu}$ of the of the network damage on the trace $\sigma$-algebra $\mathcal{F}\left(\Omega\right)\cap\left\{\mu_1<\nu<\mu\right\}$ satisfies the following formula under notation \eqref{auxfirstnotation}-\eqref{auxlasnotation}:
\begin{equation}
\varPhi_{\mu_1<\nu<\mu}=$D$_{xyw}^{-1}\left(\left[\phi_0^1-\phi_0+\frac{\varphi_0}{1-\varphi}\left(\phi^1-\phi\right)\right]\frac{\xi-\chi}{1-\psi}\right)\left(M_1,M,V\right).
\end{equation}
\end{prop}

Combining results 3.2-3.4, by linearity we can add them for the general functional.

\begin{thm}
The functional \label{f} $\varPhi_{\mu_1<\min\left\{\mu,\nu\right\}}$ of the of the network damage on the trace $\sigma$-algebra $\mathcal{F}\left(\Omega\right)\cap\left\{\mu_1<\min\left\{\mu,\nu\right\}\right\}$ satisfies the following formula under \eqref{auxfirstnotation}-\eqref{auxlasnotation}:
\begin{equation}
\varPhi_{\mu_1<\min\left\{\mu,\nu\right\}}=$D$_{xyw}^{-1}\left(\left[\phi_0^1-\phi_0+\frac{\varphi_0}{1-\varphi}(\phi^1-\phi)\right]\frac{\xi^1-\chi}{1-\psi}\right)\left(M_1,M,V\right).
\end{equation}
\end{thm}

In the next two sections, we will demonstrate analytical and numerical tractability of results derived by this approach by carrying out the inverse operators for several practical special cases. We begin with Model 1 in which the status of the network is updated upon a deterministic process, which is a reasonable assumption for many practical cases where periodic measurements are taken, as opposed to Model 2 of Section 5 in which the information about the network is collected upon random times in accordance with a specified Poisson point process (thus, with exponentially distributed inter-observation times). These two cases together represent most common types of statistical data collection.

There are other differences between the two models. In Model 1 (with ``constant observations'') we assume that the number of nodes stricken in a single attack is limited by a finite number $R$, with no further restriction upon its distribution. The associated weight of a stricken node is assumed to be gamma distributed. \label{a}In Model 2 (with ``exponential observations''), the number of nodes stricken in one attack is geometrically distributed, whereas the weight associated with each node is exponentially distributed. It seems as though the assumptions made in Model 1 are more challenging than those in Model 2, although they are rather different.

\section{\protect\centering Model 1. Constant Observations}

In this section, we will present fully explicit probabilistic results for a special case of the process with the $M_1$-auxiliary threshold (Model 1), under the following assumptions.

\begin{enumerate}
\item As previously, the attack times $\left\{t_1,t_2,...\right\}$ form a Poisson point process of rate $\lambda$.
\item Inter-observation times are constant, that is, $\Delta_k=\tau_k-\tau_{k-1}=c$ a.s., $L\left(z\right)=e^{-zc}$.
\item Nodes lost per strike have an arbitrary finite discrete
distribution, i.e.\\
$P\{n_k=j\}=p_j,\,j=1,...,R$, $g\left(z\right)=\sum\limits_{s=1}^Rp_sz^s$,$\,$ with $\boldsymbol{p}=\left(p_1,...,p_R\right)$.
\item Weight per node $w_{jk}\in\left[\text{Gamma}\left(\alpha,\xi\right)\right]$, $l\left(z\right)=\left(\frac{\xi}{z+\xi}\right)^\alpha$.
\item The initial functional $\gamma_0=1$ (i.e. zero initial damage).
\end{enumerate}

We note that deterministic observations always present a more challenging case in the context of a stochastic system than random observations. Furthermore, we restrict the total number of nodes destroyed in a single strike by a finite number $R$, which of course can be made arbitrarily large, thereby imposing literally no restriction on the distribution of perished nodes. The general gamma distribution of weight of a single node is also very general. Yet, as we will see, the main indicator of the network status will be established in a closed form functional.

To find marginal transforms we use the $\gamma$-functional of the increment of the process upon observations (as per Lemma A.1),
\begin{equation}
\gamma\left(z,v,\theta\right)=L\left[\theta+\lambda-\lambda g\left(zl\left(v\right)\right)\right]
\end{equation}
By Theorem 3.5 and due to Assumption 5,
\begin{equation}\label{constfunctional}
\varPhi_{\mu_1<\min\left\{\mu,\nu\right\}}=$D$_{xyw}^{-1}\left(\frac{\phi^1-\phi}{1-\varphi}\frac{\xi^1-\chi}{1-\psi}\right)\left(M_1,M,V\right)
\end{equation}
Furthermore,
\begin{thm}
Under Assumptions 1-5, the joint transform of the number of lost nodes, their cumulative weight, and the first passage time of the crossing of $M_1$ preceding the first crossing of $M$ or $V$ $($i.e. on the trace $\sigma$-algebra $\mathcal{F}\left(\Omega\right)\cap\left\{\mu_1<\min\left\{\mu,\nu\right\}%
\right\})$ satisfies the following formula:
\begin{align}
&E\biggl[u^{N_{\mu_1}}e^{-vW_{\mu_1}}e^{-\theta\tau_{\mu_1}}\boldsymbol{1}_{\left\{\mu_1<\min\left\{\mu,\nu\right\}\right\}}\biggr]\notag
\\&=e^{-c\left(\theta+\lambda\right)}\Biggl\{\sum\limits_{k=0}^{M_1-1}u^kF_k\left(\theta,\boldsymbol{p}\right)\sum\limits_{m=0}^{M-1-k}u^mE_m\left(\boldsymbol{p}\right)\left(\frac{\xi}{v+\xi}\right)^{\alpha\left(k+m\right)}P\left(\alpha\left(k+m\right),\left(v+\xi\right)V\right)\notag
\\&\hspace{2cm}-\,\sum\limits_{k=0}^{M_1-1}u^k\left(\frac{\xi}{\xi+v}\right)^{\alpha k}P\left(\alpha k,\left(v+\xi\right)V\right)\sum\limits_{n=0}^kE_n\left(\boldsymbol{p}\right)F_{k-n}\left(\theta,\boldsymbol{p}\right)\Biggr\},
\end{align}
\end{thm}
where
\begin{equation}
F_j\left(\theta,\boldsymbol{p}\right)=\sum\limits_{r=0}^{\left\lfloor\frac{R-1}{R}j\right\rfloor}\left(c\lambda\right)^{j-r}$Li$_{-\left(j-r\right)}\left(e^{-c\left(\theta+\lambda\right)}\right)\sum\limits_{\substack{\|\boldsymbol{\beta}\|_1=j \\ [R]\cdot\boldsymbol{\beta}=r+j}}\frac{p_1^{\beta_1}\cdots p_R^{\beta_R}}{\beta_1!\,\cdots\,\beta_R!},
\end{equation}
\begin{equation}
\left[R\right]=\left(1,...,R\right)\text{, }\boldsymbol{\beta}=\left(\beta_1,...,\beta_R\right)\in\mathbb{N}_0^R\,(\beta_j\leq R\text{,for each }j),
\end{equation}
\begin{tabbing}
\hspace{1cm}\=\hspace{0.15in}\=\kill
\LTab{Li$_s\left(z\right)=\sum\limits_{k=1}^\infty z^kk^{-s}$ is the polylogarithm,
which is numerically tractable for our\\
}\LTab{domain $\left\{e^{-w}:\text{Re}\left(w\right)>0\right\}$ with
$s\in\mathbb{Z}_{\leq 0}$,}
\end{tabbing}
\begin{equation}
E_j\left(\boldsymbol{p}\right)=\sum\limits_{r=0}^{\left\lfloor\frac{R-1}{R}j\right\rfloor}\left(c\lambda\right)^{j-r}\sum\limits_{\substack{\|\boldsymbol{\beta}\|_1=j \\ [R]\cdot\boldsymbol{\beta}=r+j}}\frac{p_1^{\beta_1}\cdots p_R^{\beta_R}}{\beta_1!\text{ }\cdots\,\beta_R!}
\end{equation}
\begin{tabbing}
\hspace{1cm}\=\kill
\LTab{$P\left(x,y\right)=1-\frac{\Gamma\left(x,y\right)}{\Gamma\left(x\right)}$ is the upper regularized gamma function,}
\end{tabbing}
\begin{tabbing}
\hspace{1cm}\=\kill
\LTab{$\Gamma\left(x,y\right)$ is the incomplete gamma function, and $\Gamma\left(x\right)$ is the gamma function.}
\end{tabbing}
\begin{proof}
By \eqref{constfunctional},
\begin{align*}
&E\left[u^{N_{\mu_1}}e^{-vW_{\mu_1}}e^{-\theta\tau_{\mu_1}}\boldsymbol{1}_{\left\{\mu_1<\min\left\{\mu,\nu\right\}\right\}}\right]
\\&=\varPhi_{\mu_1<\min\left\{\mu,\nu\right\}}\left(1,u,1,1,0,v,0,0,0,%
\theta,0,0\right)=\text{D}_{xyw}^{-1}\left(\frac{\phi^1-\phi}{1-\varphi}\frac{%
\xi^1-\chi}{1-\psi}\right)\left(M_1,M,V\right)
\end{align*}
where
\begin{align}
&\frac{\phi^1-\phi}{1-\varphi}=\frac{\gamma\left(uy,v+w,\theta\right)-%
\gamma\left(uxy,v+w,\theta\right)}{1-\gamma\left(uxy,v+w,\theta\right)}=A-B,
\\&A:=\frac{e^{-c\theta_y}}{1-e^{-c\theta_{xy}}},\label{constproof1}
\\&B:=\frac{e^{-c\theta_{xy}}}{1-e^{-c\theta_{xy}}},\label{constproof2}
\\&\frac{\xi^1-\chi}{1-\psi}=\frac{\gamma\left(1,0,0\right)-\gamma\left(y,w,0\right)}{1-\gamma\left(y,w,0\right)}=\frac{1-\gamma\left(y,w,0\right)}{1-\gamma\left(y,w,0\right)}=1,
\\&\theta_y=\theta+\lambda-\lambda\sum\limits_{s=1}^Rp_s\left(uly\right)^s,\,\theta_{xy}=\theta+\lambda-\lambda\sum\limits_{s=1}^Rp_s\left(ulyx\right)^s,\,l=l\left(v+w\right)
\end{align}
Re$\left(c\theta_{xy}\right)>0$, so $\left|e^{-c\theta_{xy}}\right|<1$, so we can find
\begin{align*}
\frac{1}{1-e^{-c\theta_{xy}}}&=\sum\limits_{i\geq 0}e^{-c\theta_{xy}i}=\sum\limits_{i\geq 0}e^{-c\left(\theta+\lambda\right)i}\sum\limits_{j\geq 0}\frac{\left(c\lambda i\right)^j}{j!}\left(\sum\limits_{s=1}^Rp_s\left(ulyx\right)^s\right)^j
\\&=\sum\limits_{j\geq 0}\frac{\left(c\lambda\right)^j}{j!}\left[\sum\limits_{i\geq 0}i^je^{-c\left(\theta+\lambda\right)i}\right]\left(\sum\limits_{k=1}^Rp_k\left(ulyx\right)^k\right)^j.
\end{align*}
By the multinomial theorem,
\begin{align*}
\hspace{1.7cm}=\sum\limits_{j\geq 0}\frac{\left(c\lambda u\right)^j}{j!}\left[\sum\limits_{i\geq
0}i^je^{-c\left(\theta+\lambda\right)i}\right]\,\,\,\,\,\sum\limits_{k=0}^{j%
\left(R-1\right)}u^kC_{jk}\left(\boldsymbol{p}\right)\left(lyx\right)^{k+j},
\end{align*}
where $C_{jk}\left(\boldsymbol{p}\right)=\sum\limits_{\substack{\left\|\beta\right\|_1=j\\\left[R\right]\cdot\boldsymbol{\beta}=k+j}}\binom{j}{\beta_1\,\cdots\,\beta_R} p_1^{\beta_1}\cdots p_R^{\beta_R}$\\
Returning to \eqref{constproof1},
\begin{equation*}
A=e^{-c\theta_y}\sum\limits_{j\geq 0}u^jA_j\left(\theta\right)\,\,\,\,\,\sum\limits_{k=0}^{j\left(R-1\right)}u^kC_{jk}\left(\boldsymbol{p}\right)\left(lyx\right)^{j+k},
\end{equation*}
where $A_k\left(\theta\right)=\frac{\left(c\lambda\right)^k}{k!}\sum\limits_{i\geq 0}i^ke^{-c\left(\theta+\lambda\right)i}$.
\begin{equation*}
A=e^{-c\theta_y}\sum\limits_{j\geq 0}\sum\limits_{k=0}^{j\left(R-1\right)}u^jA_j\left(\theta\right)u^kC_{jk}\left(\boldsymbol{p}\right)\left(lyx\right)^{j+k}.
\end{equation*}
For convenience, we denote this
\begin{equation}\label{constproof3}
A=e^{-c\theta_y}\sum\limits_{j\geq
0}\sum\limits_{k=0}^{j\left(R-1\right)}D_{jk}\left(lyx\right)^{j+k}=e^{-c\theta_y}\sum\limits_{j\geq 0}\left[\sum\limits_{\substack{m+r=j \\ r\leq m(R-1)}}D_{mr}\right]\left(lyx\right)^j.
\end{equation}
Note that
\begin{equation}\label{constproof4}
\sum\limits_{\substack{m+r=j \\ r\leq m(R-1)}}D_{mr}=\sum\limits_{r\leq\left(j-r\right)\left(R-1\right)}D_{j-r,r}=\,\,\sum\limits_{r=0}^{\left\lfloor\frac{R-1}{R}j\right\rfloor}D_{j-r,r}
\end{equation}
Combining \eqref{constproof3}-\eqref{constproof4}, we find
\begin{align}\label{constproof5}
A&=e^{-c\theta_y}\sum\limits_{j\geq 0}u^j\left[\,\,\sum\limits_{r=0}^{\left\lfloor\frac{R-1}{R}j\right\rfloor}A_{j-r}\left(\theta\right)C_{j-r,r}\left(\boldsymbol{p}\right)\right]\left(lyx\right)^j\notag
\\&=e^{-c\theta_y}\sum\limits_{j\geq
0}u^jF_j\left(\theta,\boldsymbol{p}\right)\left(lyx\right)^j
\end{align}
Applying $\mathcal{D}^{M_1-1}_x$ to \eqref{constproof5} and using properties (i, iv) of the inverse operator given in Appendix B, we get
\begin{equation}\label{constproof6}
e^{-c\theta_y}\sum\limits_{j=0}^{M_1-1}u^jF_j\left(\theta,%
\boldsymbol{p}\right)\left(ly\right)^j
\end{equation}
Applying$\mathcal{D}^{M-1}_y$ to \eqref{constproof6} with properties (\textit{i}, \textit{iii}, \textit{iv}), we have (showing only main transitions),
\begin{align}
&\sum\limits_{j=0}^{M_1-1}u^jF_j\left(\theta,\boldsymbol{p}\right)l^j%
\mathcal{D}^{M-1}_y\left(y^je^{-c\theta_y}\right)\notag
\\&=e^{-c\left(\theta+\lambda\right)}\sum\limits_{j=0}^{M_1-1}u^jF_j\left(\theta,\boldsymbol{p}\right)l^j\mathcal{D}^{M-1-j}_y\left(e^{c\lambda\sum\limits_{s=1}^Rp_s\left(uly\right)^s}\right)\notag
\\&=e^{-c\left(\theta+\lambda\right)}\sum\limits_{j=0}^{M_1-1}u^jF_j\left(\theta,\boldsymbol{p}\right)l^j\mathcal{D}^{M-1-j}_y\left(\sum\limits_{k\geq 0}u^kE_k\left(\boldsymbol{p}\right)\left(ly\right)^k\right)\notag
\\&=e^{-c\left(\theta+\lambda\right)}\sum\limits_{j=0}^{M_1-1}u^jF_j\left(\theta,\boldsymbol{p}\right)\sum\limits_{k=0}^{M-1-j}u^kE_k\left(\boldsymbol{p}\right)l^{k+j}\label{constproof7}
\end{align}
Next, applying $\mathcal{L}_w^{-1}\left(\frac{1}{w}\cdot\right)\left(V\right)$ to \eqref{constproof7},
\begin{equation}\label{constproof8}
=e^{-c\left(\theta+\lambda\right)}\,\,\,\sum\limits_{j=0}^{M_1-1}u^jF_j\left(%
\theta,\boldsymbol{p}\right)\,\,\,\sum\limits_{k=0}^{M-1-j}u^kE_k\left(%
\boldsymbol{p}\right)\left(\frac{\xi}{v+\xi}\right)^{\alpha\left(j+k\right)}P%
\left(\alpha\left(j+k\right),\left(v+\xi\right)V\right).
\end{equation}
Returning to \eqref{constproof2}, we apply the same procedure to $B$,
\begin{align}\label{constproof9}
B&=e^{-c\theta_{xy}}\sum\limits_{j\geq 0}u^jF_j\left(\theta,\boldsymbol{p}\right)\left(lyx\right)^j\notag
\\&=e^{-c\left(\theta+\lambda\right)}\left(\sum\limits_{k\geq 0}u^kE_k\left(\boldsymbol{p}\right)\left(lyx\right)^k\,\right)\left(\sum\limits_{j\geq 0}u^jF_j\left(\theta,\boldsymbol{p}\right)\left(lyx\right)^j\right)\notag
\\&=e^{-c\left(\theta+\lambda\right)}\sum\limits_{n\geq 0}u^n\left[\sum\limits_{k=0}^nE_k\left(\boldsymbol{p}\right)F_{n-k}\left(\theta,\boldsymbol{p}\right)\right]\left(lyx\right)^n.
\end{align}
Applying $\mathcal{D}^{M_1-1}_x$ and $\mathcal{D}^{M-1}_y$ to \eqref{constproof9}, we get
\begin{equation}\label{constproof10}
e^{-c\left(\theta+\lambda\right)}\,\,\,\sum\limits_{n=0}^{M_1-1}u^n\left[\sum\limits_{k=0}^nE_k\left(\boldsymbol{p}\right)F_{n-k}\left(\theta,\boldsymbol{p}\right)\right]l^n.
\end{equation}
Applying $\mathcal{L}_w^{-1}\left(\frac{1}{w}\cdot\right)\left(V\right)$ to \eqref{constproof10} yields
\begin{equation}\label{constproof11}
e^{-c\left(\theta+\lambda\right)}\sum\limits_{n=0}^{M_1-1}u^n\left[\sum\limits_{k=0}^nE_k\left(\boldsymbol{p}\right)F_{n-k}\left(\theta,\boldsymbol{p}\right)\right]\left(\frac{\xi}{\xi+v}\right)^{\alpha n}P\left(\alpha n,\left(v+\xi\right)V\right)
\end{equation}
Subtracting \eqref{constproof11} from \eqref{constproof8} yields the desired result.
\end{proof}

\begin{rem}
Through simulation of the process, we were able to produce some verification of the results via numerical examples. For two sets of parameters of the process with $R=3$ $\left(\lambda,[p_1,p_2,p_3],[\alpha,\xi],c,M_1,M,V\right)$, we generated 100 realizations of the process for each of a range of $M_1$ values and calculated the empirical probabilities $P\left\{\mu_1<\min\left\{\mu,\nu\right\}\right\}$ for each:
\end{rem}

\begin{center}
\includegraphics[scale=0.6]{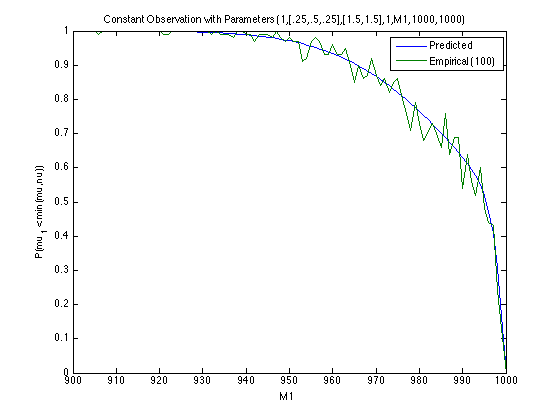}
\\
\includegraphics[scale=0.6]{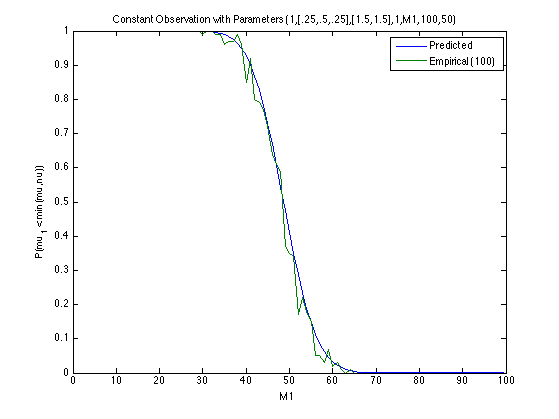}
\end{center}

\section{\protect\centering Model 2. Exponential Observations}

In this section, we will offer fully explicit formulas for another special case of the process with the $M_1$-auxiliary threshold (Model 2), under the following assumptions.

\begin{enumerate}
\item The attack times $\left\{t_1,t_2,...\right\}$ form a Poisson point process of rate $\lambda$
\item Inter-observation times $\Delta_k=\tau_k-\tau_{k-1}\in\left[\text{Exponential(}\mu\text{)}\right]$, so $L\left(z\right)=\frac{\mu}{\mu+z}$.
\item Nodes losses per strike $n_k\in\left[\text{Geometric(}a\text{)}\right]\,$(with $b=1-a$), so $g\left(z\right)=\frac{az}{1-bz}$
\item Weight lost per node $w_{jk}\in\left[\text{Exponential(}\xi\text{)}\right]$, so $l\left(z\right)=\frac{\xi}{\xi+z}$
\item The initial functional $\gamma_0=1$ (i.e. zero initial damage)
\end{enumerate}

\label{e}Using these, we find explicit marginal transforms. As in Section 4, we use the functional of the increment of the process upon observations (as shown in Lemma A.1),
\begin{equation}
\gamma\left(z,v,\theta\right)=L\left[\theta+\lambda-\lambda g\left(zl\left(v\right)\right)\right].
\end{equation}

By Theorem 3.5 and Assumption 5,
\begin{equation}\label{expfunctional}
\varPhi_{\mu_1<\min\left\{\mu,\nu\right\}}\left(u_0,u,\alpha_0,\alpha,v_0,v,\beta_0,\beta,\theta_0,\theta,h_0,h\right)=$D$_{xyw}^{-1}\left(\frac{\phi^1-\phi}{1-\varphi}\frac{\xi^1-\chi}{1-\psi}\right)\left(M_1,M,V\right).
\end{equation}

\subsection{The Joint Transform upon $\boldsymbol{\tau}_{\boldsymbol{\mu}_{\boldsymbol{1}}}$}

\noindent We will find the joint transform of the process upon the observed $M_1$-crossing and deduce some useful results from it, i.e.
\begin{equation}
\varPhi_{\mu_1<\min\left\{\mu,\nu\right\}}\left(1,u,1,1,0,v,0,0,0,\theta,0,0\right)=E\biggl[u^{N_{\mu_1}}e^{-vW_{\mu_1}}e^{-\theta\tau_{\mu_1}}\boldsymbol{1}_{\left\{\mu_1<\min\left\{\mu,\nu\right\}\right\}}\biggr]
\end{equation}

We introduce some notation convenient for the forthcoming results
\begin{equation}\label{expfirstnotation}
k_j\left(v\right)=e^{-\left(v+\xi\right)V}\sum\limits_{i=0}^j\frac{\left[\left(v+\xi\right)V\right]^i}{i!}=1-P\left(j+1,\left(v+\xi\right)V\right)\text{, }j\geq 0,
\end{equation}
\begin{align}
r_k^j\left(u,v,\theta\right)&=e^{-\left(v+\xi\right)V}\sum\limits_{i=k}^j\frac{\left[\xi d\left(\theta\right)uV\right]^i}{i!}\notag
\\&=e^{-\left(v+\xi\left(1-d\left(\theta\right)u\right)\right)V}\left[P\left(k,\xi d\left(\theta\right)uV\right)-P\left(j+1,\xi d\left(\theta\right)uV\right)\right]
\end{align}
\begin{equation}\label{explastnotation}
c\left(\theta\right)=\frac{\lambda+b\theta}{\lambda+\theta},\,d\left(\theta\right)=\frac{\lambda+b\left(\mu+\theta%
\right)}{\lambda+\mu+\theta}
\end{equation}

We will also use the convention $\sum\limits_{j=0}^{-1}=1$.

\begin{prop}
Under Assumptions 1-5 and notation \eqref{expfirstnotation}-\eqref{explastnotation}, for $M_1>M_1+1$,
\begin{align}
&E\left[u^{N_{\mu_1}}e^{-vW_{\mu_1}}e^{-\theta\tau_{\mu_1}}\boldsymbol{1}_{\left\{\mu_1<\min\left\{\mu,\nu\right\}\right\}}\right]\notag
\\&\hspace{.3cm}=\frac{a\lambda\mu\left[c\left(\theta\right)\right]^{M_1-1}\left(u\xi\right)^{M_1}}{\left(\mu+\theta+\lambda\right)\left(\theta+\lambda\right)\left(v+\xi\left(1-d\left(\theta\right)\right)\right)}\notag
\\&\hspace{.3cm}\times\left[\frac{P\left(M_1-1,\left(v+\xi\right)V\right)}{\left(v+\xi\right)^{M_1-1}}-\frac{r_{M_1-1}^{M-2}\left(u,v,\theta\right)}{\left(d\left(\theta\right)u\xi\right)^{M_1-1}}-\frac{\left(d\left(\theta\right)u\xi\right)^{M-M_1}P\left(M-1,\left(v+\xi\right)V\right)}{\left(v+\xi\right)^{M-1}}\right]
\end{align}
For $M=M_1+1$,
\begin{equation}
E\left[u^{N_{\mu_1}}e^{-vW_{\mu_1}}e^{-\theta\tau_{\mu_1}}\boldsymbol{1}_{\left\{\mu_1<\min\left\{\mu,\nu\right\}\right\}}\right]=\frac{a\lambda\mu\left[c\left(\theta\right)\right]^{M_1-1}\left(u\xi\right)^{M_1}}{\left(\mu+\theta+\lambda\right)\left(\theta+\lambda\right)\left(v+\xi\right)^{M_1}}P\left(M_1,\left(v+\xi\right)V\right)
\end{equation}
\end{prop}
\begin{proof}
By \eqref{expfunctional},
\begin{align*}
&E\biggl[u^{N_{\mu_1}}e^{-vW_{\mu_1}}e^{-\theta\tau_{\mu_1}}\boldsymbol{1}_{\left\{\mu_1<\min\left\{\mu,\nu\right\}\right\}}\biggr]
\\&\hspace{0.5cm}=\varPhi_{\mu_1<\min\left\{\mu,\nu\right\}}\left(1,u,1,1,0,v,0,0,0,\theta,0,0\right)=\text{D}_{xyw}^{-1}\left(\frac{\phi^1-\phi}{1-\varphi}\frac{\xi^1-\chi}{1-\psi}\right)\left(M_1,M,V\right).
\end{align*}
After inputting parameters and using \eqref{expfirstnotation}-\eqref{explastnotation}, we have
\begin{equation}
\frac{\phi^1-\phi}{1-\varphi}=\frac{\gamma\left(uy,v+w,\theta\right)-\gamma\left(uxy,v+w,\theta\right)}{1-\gamma\left(uxy,v+w,\theta\right)}=\frac{\mu}{\mu+\theta_y}\frac{\theta_{xy}-\theta_y}{\theta_{xy}},\hspace{.5cm}\frac{\xi^1-\chi}{1-\psi}=1,
\end{equation}
where
\begin{equation}
\theta_y=\theta+\lambda-\lambda g\left(uly\right),\hspace{.5cm}\theta_{xy}=\theta+\lambda-\lambda g\left(ulyx\right),\hspace{.5cm} l=l\left(v+w\right)
\end{equation}
Then $\mathcal{D}_x^{M_1-1}\left(\frac{\mu}{\mu+\theta_y}\frac{\theta_{xy}-\theta_y}{\theta_{xy}}\right)=\frac{\mu}{\mu+\theta_y}\mathcal{D}_x^{M_1-1}\left(\frac{\theta_{xy}-\theta_y}{\theta_{xy}}\right)$, where
\begin{equation}\label{expproof1}
\frac{\theta_{xy}-\theta_y}{\theta_{xy}}=\frac{\theta+\lambda-\lambda\frac{aulyx}{1-bulyx}-\theta-\lambda+\lambda\frac{auly}{1-buly}}{\theta+\lambda-\lambda\frac{aulyx}{1-bulyx}}=\frac{\lambda auly}{\left(\theta+\lambda\right)\left(1-buly\right)}\frac{1-x}{1-culyx}.
\end{equation}
For convenience, denote $c=c\left(\theta\right),\,d=d\left(\theta\right)$. Now we apply $\mathcal{D}_x^{M_1-1}$ to \eqref{expproof1}, using properties (i, iii, v) from Appendix B, and retain the multiplier to the left, resulting in
\begin{equation}\label{expproof2}
\frac{\mu}{\mu+\theta_y}\frac{\lambda a\left(uly\right)^{M_1}c^{M_1-1}}{\left(\theta+\lambda\right)\left(1-buly\right)}
\end{equation}
Next, we need to apply $\mathcal{D}_y^{M-1}$ to \eqref{expproof2}, so we first expand the $\frac{\mu}{\mu+\theta_y}$ term,
\begin{equation}\label{expproof3}
\frac{\mu}{\mu+\theta_y}\frac{\lambda a\left(uly\right)^{M_1}c^{M_1-1}}{\left(\theta+\lambda\right)\left(1-buly\right)}=\frac{a\mu\lambda\left(ul\right)^{M_1}c^{M_1-1}}{\left(\theta+\lambda\right)\left(\mu+\theta+\lambda\right)}\frac{y^{M_1}}{1-duly}
\end{equation}
This allows us to apply $\mathcal{D}_y^{M-1}$ to \eqref{expproof2}, via properties (i, iii, iv) of the inverse operator given in Appendix B:
\begin{equation}\label{expproof4}
\mathcal{D}_y^{M-1}\left(\frac{y^{M_1}}{1-duly}\right)=\mathcal{D}_y^{M-M_1-1}\left(\frac{1}{1-duly}\right)=\frac{1-\left(du\right)^{M-M_1}l^{M-M_1}}{1-dul}
\end{equation}
Next, we need to apply $\mathcal{L}\mathcal{C}_w^{-1}\left(\cdot\right)\left(V\right)$ to the results of \eqref{expproof3}-\eqref{expproof4}, for $M>M_1+1$:
\begin{equation}
\frac{a\mu\lambda c^{M_1-1}u^{M_1}}{\left(\theta+\lambda\right)\left(\mu+\theta+\lambda\right)}\mathcal{L}\mathcal{C}_w^{-1}\left(l^{M_1}\frac{1-\left(du\right)^{M-M_1}l^{M-M_1}}{1-dul}\right)\left(V\right).
\end{equation}
A more convenient form is as follows:
\begin{align*}
l^{M_1}\frac{1-\left(du\right)^{M-M_1}l^{M-M_1}}{1-dul}&=\left(\frac{\xi}{v+w+\xi}\right)^{M_1}\frac{1-\left(du\right)^{M-M_1}\left(\frac{\xi}{v+w+\xi}\right)^{M-M_1}}{1-du\left(\frac{\xi}{v+w+\xi}\right)}
\\&=\xi^{M_1}\left[A-\left(du\xi\right)^{M-M_1}B\right].
\end{align*}
Next, we will apply $\mathcal{L}\mathcal{C}^{-1}_w\left(\cdot\right)=\mathcal{L}^{-1}_w\left(\frac{1}{w}\cdot\right)$ to both $A$ and $B$. By partial fractions,
\begin{align*}
\mathcal{L}\mathcal{C}^{-1}_w\left(A\right)\left(V\right)=\frac{1}{v+\xi\left(1-du\right)}\Biggl[&\mathcal{L}^{-1}_w\left(\frac{1}{w\left(w+v+\xi\right)^{M_1-1}}\right)\left(V\right)
\\&-\mathcal{L}^{-1}_w\left(\frac{1}{\left(w+v+\xi\left(1-du\right)\right)\left(w+v+\xi\right)^{M_1-1}}\right)\left(V\right)\Biggr]
\end{align*}
\begin{equation}\label{expproof5}
=\frac{1}{v+\xi\left(1-du\right)}\Biggl[\frac{1-k_{M_1-2}\left(v\right)}{\left(v+\xi\right)^{M_1-1}}-\frac{e^{-\left(v+\xi\left(1-du\right)\right)V}}{\left(du\xi\right)^{M_1-1}}\left(1-e^{-\xi duV}\sum\limits_{j=0}^{M_1-2}\frac{\left(\xi duV\right)^j}{j!}\right)\Biggr].
\end{equation}
The inverse transform of $B$ is similar, where $M_1\,$is replaced by $M$. Combining this inverse with \eqref{expproof5} yields the statement of the theorem for $M>M_1+1$.

For $M=M_1+1$, the inverse Laplace-Carson transform reduces to\\

\hspace{.5cm}$\mathcal{L}\mathcal{C}_w^{-1}\left(l^{M_1}\right)\left(V\right)=\mathcal{L}^{-1}_w\left(\frac{\xi^{M_1-1}}{w\left(w+v+\xi\right)^{M_1-1}}\right)\left(V\right)=\left(\frac{\xi}{v+\xi}\right)^{M_1-1}\left[1-k_{M_1-2}\left(v\right)\right]$
\end{proof}

The first fact we deduce from the joint transform above is the probability that the observed $M_1$ crossing occurs before crossing of an observed $M$ or $V$.

\begin{cor}
Under Assumptions 1-5 and notation \eqref{expfirstnotation}-\eqref{explastnotation} for $M>M_1+1$,
\begin{align}
&\varPhi_{\mu_1<\min\left\{\mu,\nu\right\}}\left(1,1,1,1,0,0,0,0,0,0,0,0\right)=E\left[\boldsymbol{1}_{\left\{\mu_1<\min\left\{\mu,\nu\right\}\right\}}\right]=P\left\{\mu_1<\min\left\{\mu,\nu\right\}\right\}\notag
\\&=P\left(M_1-1,\xi V\right)-\frac{r_{M_1-1}^{M-2}\left(1,0,0\right)}{d\left(0\right)^{M_1-1}}-d\left(0\right)^{M-M_1}P\left(M-1,\xi V\right).
\end{align}
\end{cor}

\begin{rem}
Through simulation of the process, we were able to produce some verification of the results via numerical examples. For two sets of parameters of the process $\left(\lambda,a,\xi,\mu,M_1,M,V\right)$, we generated 100 realizations of the process for each of a range of $M_1$ values and calculated the empirical probabilities of $P\left\{\mu_1<\min\left\{\mu,\nu\right\}\right\}$ for each:
\end{rem}

\begin{center}
\includegraphics[scale=0.6]{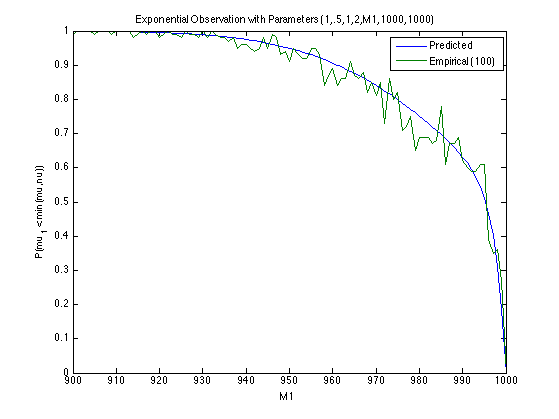}
\\
\includegraphics[scale=0.6]{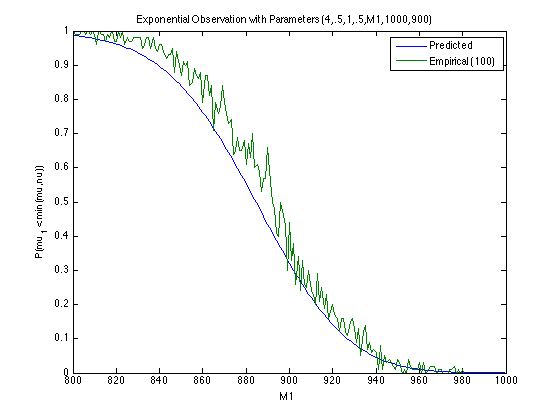}
\end{center}

\subsection{Marginal Transforms upon $\boldsymbol{\tau}_{\boldsymbol{\mu}_{\boldsymbol{1}}}$}

\noindent In this subsection, we present the joint transform and marginal transforms of each component of the process upon the (intermediate) control level $M_1$ crossing. These follow trivially from the joint functional of the previous subsection by adjusting the values of $u,\,v,\,$and $\theta$.
\begin{cor}
Under Assumptions 1-5 and notation
\eqref{expfirstnotation}-\eqref{explastnotation} for $M>M_1+1$,
\begin{align}
&\hspace{-.85cm}\varPhi_{\mu_1<\min\left\{\mu,\nu\right\}}\left(1,u,1,1,0,0,0,0,0,0,0,0\right)=E\left[u^{N_{\mu_1}}\boldsymbol{1}_{\left\{\mu_1<\min\left\{\mu,\nu\right\}\right\}}\right]\notag
\\&\hspace{-.85cm}=\frac{a\mu u^{M_1}}{\left(\mu+\lambda\right)\left(1-d\left(0\right)u\right)}\left[P\left(M_1-1,\xi V\right)-\frac{r_{M_1-1}^{M-2}\left(u,0,0\right)}{\left(d\left(0\right)u\right)^{M_1-1}}-\left(d\left(0\right)u\right)^{M-M_1}P\left(M-1,\xi V\right)\right].
\end{align}
\end{cor}
\begin{cor}
Under Assumptions 1-5 and notation \eqref{expfirstnotation}\eqref{explastnotation}, for $M>M_1+1$,
\begin{align}
&\varPhi_{\mu_1<\min\left\{\mu,\nu\right\}}\left(1,1,1,1,0,v,0,0,0,0,0,0\right)=E\left[e^{-vW_{\mu_1}}\boldsymbol{1}_{\left\{\mu_1<\min\left\{\mu,\nu\right\}\right\}}\right]\notag
\\&=\frac{a\mu}{\mu+\lambda}\frac{\xi^{M_1}}{v+\xi\left(1-d\left(0\right)\right)}\Biggl[\frac{P\left(M_1-1,\left(\xi+v\right)V\right)}{\left(v+\xi\right)^{M_1-1}}-\frac{r_{M_1-1}^{M-2}\left(1,v,0\right)}{\left(d\left(0\right)\xi\right)^{M_1-1}}\notag
\\&\hspace{4cm}-\frac{\left(d\left(0\right)\xi\right)^{M-M_1}P\left(M-1,\left(\xi+v\right)V\right)}{\left(v+\xi\right)^{M-1}}\Biggr].
\end{align}
\end{cor}
\begin{cor}
Under Assumptions 1-5 and notation \eqref{expfirstnotation}-\eqref{explastnotation}, for $M>M_1+1$,
\begin{align}
&\varPhi_{\mu_1<\min\left\{\mu,\nu\right\}}\left(1,1,1,1,0,0,0,0,0,\theta,0,0\right)=E\left[e^{-\theta\tau_{\mu_1}}\boldsymbol{1}_{\left\{\mu_1<\min\left\{\mu,\nu\right\}\right\}}\right]\notag
\\&=\frac{a\mu\lambda c\left(\theta\right)^{M_1-1}}{\left(\theta+\lambda\right)\left(\mu+\theta+%
\lambda\right)\left(1-d\left(\theta\right)\right)}\notag
\\&\hspace{.5cm}\times\Biggl[P\left(M_1-1,\xi V\right)-\frac{r_{M_1-1}^{M-2}\left(1,0,\theta\right)}{d\left(\theta\right)^{M-M_1}}-d\left(\theta\right)^{M-M_1}P\left(M-1,\xi V\right)\Biggr].
\end{align}
\end{cor}

\subsection{The Joint Transform upon $\boldsymbol{\tau}_{\min\left\{\boldsymbol{\mu,\nu}\right\}}$}

\noindent In this subsection, we will find the joint transform of the process upon the first observed crossing of $M$ or $V$. By \eqref{expfunctional}, the desired transform is
\begin{align}
&\varPhi_{\mu_1<\min\left\{\mu,\nu\right\}}\left(1,1,1,\alpha,0,0,0,\beta,0,0,0,h\right)\notag
\\&=E\biggl[\alpha^{N_{\min\left\{\mu,\nu\right\}}}e^{-\beta W_{\min\left\{\mu,\nu\right\}}}e^{-h\tau_{\min\left\{\mu,\nu\right\}}}\boldsymbol{1}_{\left\{\mu_1<\min\left\{\mu,\nu\right\}\right\}}\biggr].
\end{align}
We will use the following additional notation,
\begin{equation}\label{expnotation2}
s\left(x,y\right)=\sum\limits_{i=1}^{M-M_1-1}\left(\frac{c\left(x\right)}{d\left(y\right)}\right)^i
\end{equation}

\begin{prop}
Under Assumptions 1-5 and notation \eqref{expfirstnotation}-\eqref{explastnotation} and \eqref{expnotation2}, for $M>M_1+1$,
\begin{align}
&E\biggl[\alpha^{N_{\min\left\{\mu,\nu\right\}}}e^{-\beta W_{\min\left\{\mu,\nu\right\}}}e^{-h\tau_{\min\left\{\mu,\nu\right\}}}\boldsymbol{1}_{\left\{\mu_1<\min\left\{\mu,\nu\right\}\right\}}\biggr]\notag
\\&=\frac{\left(a\mu\lambda\right)^2\alpha^{M_1+1}\xi^{M_1}c\left(h\right)^{M_1-1}}{\left(1-d\left(h\right)\alpha l\left(\beta\right)\right)\left(h+\lambda\right)^2\left(\mu+h+\lambda\right)^2}\frac{1}{\beta+\xi\left(1-d\left(h\right)\alpha\right)}\notag
\\&\hspace{.5cm}\times\Biggl\{l\left(\beta\right)\left[\frac{P\left(M_1-1,\left(\beta+\xi\right)V\right)}{\left(\beta+\xi\right)^{M_1-1}}-\frac{e^{-\left(\beta+\xi\left(1-d\left(h\right)\alpha\right)\right)V}-r_0^{M_1-2}\left(\alpha,\beta,h\right)}{\left(d\left(h\right)\alpha\xi\right)^{M_1-1}}\right]\notag
\\&\hspace{.5cm}+\left(l\left(\beta\right)-\frac{1}{\alpha
c}\right)\sum\limits_{j=1}^{M-M_1-1}\left(c\left(h\right)\alpha\xi\right)^j\notag
\\&\hspace{1cm}\times\left[\frac{P\left(M_1+j-1,\left(\beta+\xi\right)V\right)}{\left(\beta+\xi\right)^{M_1+j-1}}-\frac{e^{-\left(\beta+\xi\left(1-d\left(h\right)\alpha\right)\right)V}-r_0^{M_1+j-2}\left(\alpha,\beta,h\right)}{\left(d\left(h\right)\alpha\xi\right)^{M_1+j-1}}\right]\notag
\\&\hspace{.5cm}-\left(d\left(h\right)\alpha\xi\right)^{M-M_1}\left[l\left(\beta\right)+\left(l\left(\beta\right)-\frac{1}{\alpha c\left(h\right)}\right)s\left(h,h\right)\right]\notag
\\&\hspace{1cm}\times\left[\frac{P\left(M-1,\left(\beta+\xi\right)V\right)}{\left(\beta+\xi\right)^{M-1}}-\frac{e^{-\left(\beta+\xi\left(1-d\left(h\right)\alpha\right)\right)V}-r_0^{M-2}\left(\alpha,\beta,h\right)}{\left(d\left(h\right)\alpha\xi\right)^{M-1}}\right]\Biggr\}.
\end{align}
\end{prop}
\begin{proof}
The following term is independent of $x$.
\begin{align}
&\frac{\xi^1-\chi}{1-\psi}=\frac{\gamma\left(\alpha,w,h\right)-\gamma\left(\alpha y,w+\beta,h\right)}{1-\gamma\left(\alpha y,w+\beta,h\right)}=\frac{\mu}{\mu+h_\ast}\frac{h_y-h_\ast}{h_y},
\\&h_y=h+\lambda-\lambda\frac{a\alpha ly}{1-b\alpha ly},\hspace{.5cm}h_\ast=h+\lambda-\lambda\frac{a\alpha l\left(\beta\right)}{1-b\alpha l\left(\beta\right)},\hspace{.5cm}l=l\left(w+\beta\right),
\end{align}
so as in \eqref{expproof1}-\eqref{expproof4} with $\left(u,v,\theta\right)$ replaced with $\left(\alpha,\beta,h\right)$, we apply $\mathcal{D}_x^{M_1-1}$ to $\frac{\phi^1-\phi}{1-\varphi}$ while denoting $c=c\left(h\right),\,d=d\left(h\right)$ for convenience,
\begin{equation}
\mathcal{D}_x^{M_1-1}\left(\frac{\phi^1-\phi}{1-\varphi}\right)=\frac{\mu}{\mu+h_y}\frac{a\lambda\left(\alpha l\right)^{M_1}c^{M_1-1}}{h+\lambda}\frac{y^{M_1}}{1-b\alpha ly}.
\end{equation}
Manipulating $\frac{\mu}{\mu+h_y}$, $\frac{h_y-h_\ast}{h_y}$, and
$\frac{\mu}{\mu+h_\ast}$ similarly and combining with the above yields
\begin{equation}\label{expproof6}
\frac{\xi^1-\chi}{1-\psi}\mathcal{D}_x^{M_1-1}\left(\frac{\phi^1-\phi}{1-\varphi}\right)=Cl^{M_1}\frac{y^{M_1}\left(l\left(\beta\right)-ly\right)}{\left(1-c\alpha ly\right)\left(1-d\alpha ly\right)}
\end{equation}
where $C=\frac{\left(a\mu\lambda\right)^2\alpha^{M_1+1}c^{M_1-1}}{\left(1-d\alpha l\left(\beta\right)\right)\left(h+\lambda\right)^2\left(\mu+h+\lambda\right)^2}$. We can now apply $\mathcal{D}_y^{M-1}$ to \eqref{expproof6}, using properties (\textit{i}, \textit{iii}, \textit{vi}) of the inverse operator as given in Appendix B,
\begin{align}\label{expproof7}
&Cl^{M_1}\mathcal{D}_y^{M-1}\left(\frac{l\left(\beta\right)y^{M_1}-ly^{M_1+1}}{\left(1-c\alpha ly\right)\left(1-d\alpha ly\right)}\right)\notag
\\&=Cl\left(\beta\right)\frac{l^{M_1}}{1-d\alpha l}+C\left(l\left(\beta\right)-\frac{1}{\alpha c}\right)\sum\limits_{j=1}^{M-M_1-1}\left(c\alpha\right)^j\frac{l^{M_1+j}}{1-d\alpha l}\notag
\\&\hspace{.5cm}-C\left(d\alpha\right)^{M-M_1}\left[l\left(\beta\right)+\left(l\left(\beta\right)-\frac{1}{\alpha c}\right)s\left(h,h\right)\right]\frac{l^M}{1-d\alpha l}.
\end{align}
Next, we need to apply $\mathcal{L}\mathcal{C}^{-1}_w\left(\cdot\right)\left(V\right)$ to \eqref{expproof7}. Notice that each non-constant term (with respect to $w$) is of the form $\frac{l^k}{1-d\alpha l}$, so we establish a result for an arbitrary $k$. We have
\begin{equation*}
\frac{1}{w}\cdot\frac{l^k}{1-d\alpha l}=\frac{\xi^k}{w\left(w+\beta+\xi\right)^{k-1}\left(w+\beta+\xi\left(1-d\alpha\right)\right)}.
\end{equation*}
Applying $\mathcal{L}\mathcal{C}_w^{-1}$ here is the same as in \eqref{expproof5} with $M_1=k$, $v=\beta$, and $\theta=h$, so we have
\begin{align}
\mathcal{L}^{-1}_w\left(\frac{1}{w}\frac{l^k}{1-d\alpha l}\right)\left(V\right)&=\frac{\xi^k}{\beta+\xi\left(1-d\alpha\right)}\notag
\\&\hspace{.5cm}\times\left[\frac{1-k_{k-2}\left(\beta\right)}{\left(\beta+\xi\right)^{k-1}}-\frac{e^{-\left(\beta+\xi\left(1-d\left(h\right)\alpha\right)\right)V}-r_0^{k-2}\left(\alpha,\beta,h\right)}{\left(d\alpha\xi\right)^{k-1}}\right].
\end{align}
Applying this result to \eqref{expproof7} and expanding the constant $C$ yields the statement of the theorem for $M>M_1+1$. Deriving a formula for $M=M_1+1$ is similar to the second part of the proof of Proposition 5.1, where applying $\mathcal{D}_y^{M_1}$ instead yields a constant, $l\left(\beta\right)$ in \eqref{expproof7}.
\end{proof}

\subsection{Marginal Transforms upon $\boldsymbol{\tau}_{\min\left\{\boldsymbol{\mu,\nu}\right\}}$}

\noindent Next, we find the marginal transforms of each component upon $\tau_{\min\left\{\mu,\nu\right\}}$, the first observed crossing of $M$ or $V$. These readily follow from the joint functional of the previous section by adjusting the values of $\alpha,\,\beta,\,$and $h$.
\begin{cor}
Under Assumptions 1-5 and notation \eqref{expfirstnotation}-\eqref{explastnotation} and \eqref{expnotation2},
\begin{align}
&E\left[\alpha^{N_{\min\left\{\mu,\nu\right\}}}\boldsymbol{1}_{\left\{\mu_1<\min\left\{\mu,\nu\right\}\right\}}\right]\notag
\\&=\frac{\left(a\mu\right)^2\alpha^{M_1+1}}{\left(1-d\left(0\right)\alpha\right)^2\left(\mu+h+\lambda\right)^2}\Biggl\{P\left(M_1-1,\xi V\right)-\frac{e^{-\xi\left(1-d\left(0\right)\alpha\right)V}-r_0^{M_1-2}\left(\alpha,0,0\right)}{\left(d\left(0\right)\alpha\right)^{M_1-1}}\notag
\\&\hspace{.5cm}+\frac{\alpha-1}{\alpha}\sum\limits_{j=1}^{M-M_1-1}\alpha^j\left[P\left(M_1+j-1,\xi V\right)-\frac{e^{-\xi\left(1-d\left(h\right)\alpha\right)V}-r_0^{M_1+j-2}\left(\alpha,0,0\right)}{\left(d\left(0\right)\alpha\right)^{M_1+j-1}}\right]\notag
\\&\hspace{.5cm}-\left(d\left(0\right)\alpha\xi\right)^{M-M_1}\left[1+\left(\frac{\alpha-1}{\alpha}\right)s\left(0,0\right)\right]\notag
\\&\hspace{1cm}\times\left[P\left(M-1,\xi V\right)-\frac{e^{-\xi\left(1-d\left(0\right)\alpha\right)V}-r_0^{M-2}\left(\alpha,0,0\right)}{\left(d\left(0\right)\alpha\right)^{M-1}}\right]\Biggr\}.
\end{align}
\end{cor}
\begin{cor}
Under Assumptions 1-5 and notation \eqref{expfirstnotation}-\eqref{explastnotation} and \eqref{expnotation2},
\begin{align}
&\hspace{-.7cm}E\left[e^{-\beta W_{\min\left\{\mu,\nu\right\}}}\boldsymbol{1}_{\left\{\mu_1<\min\left\{\mu,\nu\right\}\right\}}\right]\notag
\\&\hspace{-.7cm}=\frac{\left(a\mu\right)^2\xi^{M_1}}{\left(1-d\left(0\right)l\left(\beta\right)\right)\left(\mu+\lambda\right)^2}\frac{1}{\beta+\xi\left(1-d\left(0\right)\right)}\notag
\\&\hspace{-.2cm}\times\Biggl\{l\left(\beta\right)\left[\frac{P\left(M_1-1,\left(\beta+\xi\right)V\right)}{\left(\beta+\xi\right)^{M_1-1}}-\frac{e^{-\left(\beta+\xi\left(1-d\left(h\right)\right)\right)V}-r_0^{M_1-2}\left(1,\beta,0\right)}{\left(d\left(0\right)\xi\right)^{M_1-1}}\right]\notag
\\&\hspace{.3cm}+\left(l\left(\beta\right)-1\right)\sum\limits_{j=1}^{M-M_1-1}\xi^j\left[\frac{P\left(M_1+j-1,\left(\beta+\xi\right)V\right)}{\left(\beta+\xi\right)^{M_1+j-1}}-\frac{e^{-\left(\beta+\xi\left(1-d\left(0\right)\right)\right)V}-r_0^{M_1+j-2}\left(1,\beta,0\right)}{\left(d\left(0\right)\xi\right)^{M_1+j-1}}\right]\notag
\\&\hspace{.3cm}-\left(d\left(0\right)\xi\right)^{M-M_1}\left[l\left(\beta\right)+\left(l\left(\beta\right)-1\right)s\left(0,0\right)\right]\notag
\\&\hspace{.7cm}\times\left[\frac{P\left(M-1,\left(\beta+\xi\right)V\right)}{\left(\beta+\xi\right)^{M-1}}-\frac{e^{-\left(\beta+\xi\left(1-d\left(h\right)\right)\right)V}-r_0^{M-2}\left(1,0,0\right)}{\left(d\left(h\right)\xi\right)^{M-1}}\right]\Biggr\}.
\end{align}
\end{cor}

\begin{cor}
Under Assumptions 1-5 and notation \eqref{expfirstnotation}-\eqref{explastnotation} and \eqref{expnotation2},
\begin{align}
&E\left[e^{-h\tau_{\min\left\{\mu,\nu\right\}}}\boldsymbol{1}_{\left\{\mu_1<\min\left\{\mu,\nu\right\}\right\}}\right]\notag
\\&=\frac{\left(a\mu\lambda\right)^2c\left(h\right)^{M_1-1}}{\left(1-d\left(h\right)\right)^2\left(h+\lambda\right)^2\left(\mu+h+\lambda\right)^2}\Biggl\{\left[P\left(M_1-1,\xi V\right)-\frac{e^{-\xi\left(1-d\left(h\right)\right)V}-r_0^{M_1-2}\left(1,0,h\right)}{d\left(h\right)^{M_1-1}}\right]\notag
\\&\hspace{.5cm}+\frac{c\left(h\right)-1}{c\left(h\right)}\sum\limits_{j=1}^{M-M_1-1}c\left(h\right)^j\left[1-P\left(M_1+j-1,\xi V\right)-\frac{e^{-\xi\left(1-d\left(h\right)\right)V}-r_0^{M_1+j-2}\left(1,0,h\right)}{d\left(h\right)^{M_1+j-1}}\right]\notag
\\&\hspace{.5cm}-d\left(h\right)^{M-M_1}\left[1+\left(\frac{c\left(h\right)-1}{c\left(h\right)}\right)s\left(h,h\right)\right]\notag
\\&\hspace{1cm}\times\left[P\left(M-1,\xi V\right)-\frac{e^{-\xi\left(1-d\left(h\right)\right)V}-r_0^{M-2}\left(1,0,h\right)}{d\left(h\right)^{M-1}}\right]\Biggr\}.
\end{align}
\end{cor}

\subsection{Time Between Successive Threshold Crossings}

\noindent Next, we provide a functional of the time between the observed $M_1$ crossing and the first observed $M$ or $V$ crossing,
\begin{equation}
\varPhi_{\mu_1<\min\left\{\mu,\nu\right\}}\left(1,1,1,1,0,0,0,0,0,-h,0,h\right)=E\left[e^{-h\left(\tau_{\min\left\{\mu,\nu\right\}}-\tau_{\mu_1}\right)}\boldsymbol{1}_{\left\{\mu_1<\min\left\{\mu,\nu\right\}\right\}}\right].
\end{equation}

\begin{prop}
Under Assumptions 1-5 and notation \eqref{expfirstnotation}-\eqref{explastnotation} and \eqref{expnotation2},
\begin{align}
&E\left[e^{-h\left(\tau_{\min\left\{\mu,\nu\right\}}-\tau_{\mu_1}\right)}\boldsymbol{1}_{\left\{\mu_1<\min\left\{\mu,\nu\right\}\right\}}\right]\notag
\\&=\frac{a\mu^2\lambda}{\left(\mu+\lambda\right)\left(\mu+h\right)\left(h+\lambda\right)\left(1-d\left(0\right)\right)}\Biggl\{\left[P\left(M_1-1,\xi V\right)-\frac{e^{-\xi\left(1-d\left(0\right)\right)V}-r_0^{M_1-2}\left(1,0,0\right)}{d\left(0\right)^{M_1-1}}\right]\notag
\\&\hspace{.5cm}+\frac{c\left(h\right)-1}{c\left(h\right)}\sum\limits_{j=1}^{M-M_1-1}c\left(h\right)^j\left[P\left(M_1+j-1,\xi V\right)-\frac{e^{-\xi\left(1-d\left(0\right)\right)V}-r_0^{M_1+j-2}\left(1,0,0\right)}{d\left(0\right)^{M_1+j-1}}\right]\notag
\\&\hspace{.5cm}-d\left(0\right)^{M-M_1}\left(1+\frac{c\left(h\right)-1}{c\left(h\right)}s\left(h,0\right)\right)\notag
\\&\hspace{1cm}\times\left[P\left(M-1,\xi V\right)-\frac{e^{-\xi\left(1-d\left(0\right)\right)V}-r_0^{M-2}\left(1,0,0\right)}{d\left(0\right)^{M-1}}\right]\Biggr\}.
\end{align}
\end{prop}
\noindent The proof is similar to the above propositions and thus is omitted.\hfill{\qedsymbol}

\subsection{Summary}

\noindent In this article we deal with the detection and prediction of losses due to cyber attacks waged on large-scale vital networks. We model the accumulation of losses to a network during a series of hostile attacks by a 2-dimensional monotone random walk process as observed by an independent delayed renewal process. The first component of the process is associated with the number of nodes (such as routers or operational sites) incapacitated by successive attacks. Each node has a weight associated with its incapacitation (such as loss of operational capacity or financial expense associated with repair), and the second component represents the cumulative weight associated with the nodes lost. Each component has a fixed threshold, and crossing of a threshold by either component represents the network entering a critical condition. We obtain tractable results in the form of joint transforms of the predicted time of the first observed threshold crossing, along with the values of each underlying component (such as the number of perished nodes, their associated cumulative weight) upon this time and one observation prior to the first passage time. We demonstrated the tractability of the obtained results on two major and various other special cases. We further validated the results through the comparison with stochastically simulated attacks.

\subsection{A Forthcoming and Future Work.}

\noindent We continue our research on strategic networks in two directions. First, we plan to further refine the obtained predicted data hindered by an eventual crudeness of observations. This will be rendered by means of so-called ``time sensitive analysis.'' The latter is a completely different technique applied to continuous time parameter processes and it will allow us to obtain the underlying distributions in any vicinity of the first passage time reviving missed moments of attacks due to delayed observations. These measures may yield a more accurate information about the status of the network than those delivered by auxiliary thresholds. The method of auxiliary thresholds will thus be replaced by the time sensitive analysis, although a combination of both techniques is not excluded. Independent of this refinement, we also plan to introduce strategic defense of finite networks, which can be utilized to random walks on smaller scale graphs.

\section*{Acknowledgement}

\noindent The authors are very grateful to the referees for their constructive remarks and suggestions that led to a better version.

\bibliographystyle{plain}
\bibliography{OSDSN}
\nocite{Agarwal2004,Akyidiz2002,Barrat2008,Bingham2001,
Boginski2009,Crescenzo2009,
Dshalalow2005,Dshalalow2006,Dshalalow2009,Elsasser2006,Franceschetti2007,
Franceschetti2007a,Garlaschelli2009,Gilbert1961,Haenggi2009,Hida1995,
Kadankov2005,Kadankova2007,Kalisky2006,Kyprianou2003,Mellander1992,
Muzy2000,Newman2002,Newman2002a,Newman2003,Newman2004,Porfiri2008,
Redner2001,Takacs1978,Takacs1984,Telcs1989,Wu1982}
\appendix
\begin{appendices}
\section{\protect\centering Proof of the Value of $\gamma(z,v,\theta)$}
\begin{lem}
The joint functional of increments per observation epoch is
\begin{equation}
\gamma\left(z,v,\theta\right)=L\left[\theta+\lambda-\lambda g\left(zl\left(v\right)\right)\right]
\end{equation}
\end{lem}
\begin{proof}
\begin{align}
\gamma\left(z,v,\theta\right)&=E\left[z^{X_1}e^{-vY_1}e^{-\theta\Delta_1}\right]=E\left[e^{-\theta\Delta_1}E\left[z^{X_1}e^{-vY_1}\Big|\Delta_1\right]\right]\notag
\\&=E\left[e^{-\theta\Delta_1}E\left[\left(z^{n_1}e^{-vw_{11}}\right)\cdots\left(z^{n_J}e^{-vw_J}\right)\Big|\Delta_1\right]\right]\notag
\\&=E\Biggl[e^{-\theta\Delta_1}E\Biggl[z^{n_1}e^{-v\left(w_{11}+...+w_{n_11}\right)}\times\cdots\times z^{nJ}e^{-v\left(w_{1J}+...+w_{n_JJ}\right)}\Bigg|\Delta_1\Biggr]\Biggr]\notag
\\&=E\left[e^{-\theta\Delta_1}E\left[E\left[\left(zl\left(v\right)\right)^{n_1}\big|n_1\right]^J\Big|\Delta_1\right]\right]=E\left[e^{-\theta\Delta_1}E\left[\left(g\left(zl\left(v\right)\right)\right)^J\Big|\Delta_1\right]\right]\notag
\\&=L\left[\theta+\lambda-\lambda g\left(zl\left(v\right)\right)\right]
\end{align}
where $J$ is the number of strikes arriving in $\Delta_1$, which is a Poisson RV with parameter $\lambda\Delta_1$.
\end{proof}

\section{\protect\centering Properties of the Inverse Operator $\mathcal{D}^k$}	

In this appendix, we mention some useful properties of the inverse operator $\mathcal{D}^k$, as defined in \eqref{doperatorinversebasic}.

\begin{enumerate}[\itshape($i$)]
\item $\mathcal{D}^k$ is a linear functional.
\item $\mathcal{D}^k_x\left(\boldsymbol{1}\left(x\right)\right)=1$, where \textbf{1}$\left(x\right)=1$ for all $x\in\mathbb{R}$.
\item If $g$ is an analytic function at $0$, then
\begin{equation*}
\mathcal{D}^k_x\left(x^jg\left(x\right)\right)=\mathcal{D}_x^{k-j}\left(g\left(x\right)\right).
\end{equation*} 
\item If $a\left(x\right)=\sum\limits_{j=0}^\infty a_jx^j$, then
\begin{equation*}
\mathcal{D}^k_x\left(a\left(xy\right)\right)=\sum\limits_{j=0}^ka_jy^j.
\end{equation*}
\item If $b\neq 1$,$\,$for $b\in\mathbb{R}$, then
\begin{equation*}
\mathcal{D}_x^k\left(\frac{1}{1-bx}\right)=\frac{1-b^{k+1}}{1-b}.
\end{equation*}
\item If
$b\neq 1$,$\,$for $a,b\in\mathbb{R}$, then
\begin{equation*}
\mathcal{D}_x^k\left(\frac{1}{1-bx}\frac{1}{1-ax}\right)=\frac{1}{1-b}\sum\limits_{j=0}^ka^j\left(1-b^{k+1-j}\right).
\end{equation*}
\end{enumerate}
\end{appendices}
\end{document}